\documentclass[11pt]{amsart}
\usepackage{amsmath}
\usepackage{amssymb}
\usepackage{amsthm}
\usepackage{latexsym}
\usepackage{graphicx}
\usepackage{hyperref}
\usepackage{enumerate}
\usepackage[all]{xy}
\usepackage{chngcntr}

\newtheorem{thm}{Theorem}[section]

\newtheorem{lem}[thm]{Lemma}
\newtheorem{prop}[thm]{Proposition}

\newtheorem{thmintro}{Theorem}

\theoremstyle{definition}
\newtheorem{defn}[thm]{Definition}
\newtheorem{rem}[thm]{Remark}

\newcommand{\N}{\mathbb N}
\newcommand{\Z}{\mathbb Z}
\newcommand{\Q}{\mathbb Q}
\newcommand{\R}{\mathbb R}
\newcommand{\C}{\mathbb C}
\newcommand{\F}{\mathbb F}
\newcommand{\mf}{\mathfrak}
\newcommand{\mc}{\mathcal}
\newcommand{\mb}{\mathbf}
\newcommand{\mh}{\mathbb}
\def\cC{{\mathcal C}}

\def\cV{{\mathcal V}}
\def\cE{{\mathcal E}}
\def\Irr{{\rm Irr}}
\newcommand{\mr}{\mathrm}
\newcommand{\ind}{\mathrm{ind}}
\newcommand{\enuma}[1]{\begin{enumerate}[\textup{(}a\textup{)}] {#1} \end{enumerate}}
\newcommand{\Fr}{\mathrm{Frob}}

\newcommand{\ad}{\mathrm{ad}}

\newcommand{\Rep}{\mathrm{Rep}}
\newcommand{\Res}{\mathrm{Res}}

\newcommand{\der}{\mathrm{der}}

\newcommand{\Mod}{\mathrm{Mod}}
\newcommand{\Hom}{\mathrm{Hom}}
\newcommand{\End}{\mathrm{End}}

\newcommand{\sgn}{\mathrm{sgn}}
\newcommand{\pt}{\mathrm{pt}}

\newcommand{\Modf}[1]{\mathrm{Mod}_{\mathrm{fl}, #1}}
\newcommand{\Ad}{\mathrm{Ad}}
\newcommand{\pr}{\mathrm{pr}}
\newcommand{\reg}{\mathrm{reg}}
\newcommand{\IC}{\mathrm{IC}}
\newcommand{\Ql}{\overline{\mathbb Q_\ell}}
\newcommand{\Ext}{\mathrm{Ext}}

\begin{document}

\title[Hecke algebras and constructible sheaves on the nilpotent cone]{Graded 
Hecke algebras and equivariant constructible sheaves on the nilpotent cone}
\date{\today}
\maketitle

\vspace{5mm}

\begin{center}
{\Large Maarten Solleveld} \\[1mm]
IMAPP, Radboud Universiteit Nijmegen\\
Heyendaalseweg 135, 6525AJ Nijmegen, the Netherlands \\
email: m.solleveld@science.ru.nl 
\end{center}
\vspace{5mm}

\begin{abstract}
Graded Hecke algebras can be constructed geometrically, with constructible sheaves and equivariant
cohomology. The input consists of a complex reductive group $G$ (possibly disconnected) and a cuspidal 
local system on a nilpotent orbit for a Levi subgroup of $G$. We prove that every such ``geometric" 
graded Hecke algebra is naturally isomorphic to the endomorphism algebra of a certain 
$G \times \C^\times$-equivariant semisimple complex of sheaves on the nilpotent cone $\mf g_N$ 
in the Lie algebra of $G$.

From there we provide an algebraic description of the $G \times \C^\times$-equivariant bounded 
derived category of constructible sheaves on $\mf g_N$. Namely, it is equivalent with the bounded
derived category of finitely generated differential graded modules of a suitable direct sum of
graded Hecke algebras. This can be regarded as a categorification of graded Hecke algebras.

This paper prepares for a study of representations of reductive $p$-adic groups with
a fixed infinitesimal central character. In sequel papers \cite{SolSheaves,SolStand}, that will lead 
to proofs of the generalized injectivity conjecture and of the Kazhdan--Lusztig conjecture for 
$p$-adic groups. 
\end{abstract}

\vspace{5mm}

\tableofcontents

\section*{Introduction}

The story behind this paper started with with the seminal work of Kazhdan and Lusztig \cite{KaLu}.
They showed that an affine Hecke algebra $\mc H$ is naturally isomorphic with a $K$-group of equivariant 
coherent sheaves on the Steinberg variety of a complex reductive group. (Here $\mc H$ has a formal 
variable $\mb q$ as single parameter and the reductive group must have simply connected derived group.)
This isomorphism enables one to regard the category of equivariant coherent sheaves on that particular 
variety as a categorification of an affine Hecke algebra. Later that became quite an important theme 
in the geometric Langlands program, see for instance \cite{Bez,ChGi,MiVi}.

Our paper is inspired by the quest for a generalization of such a categorification of $\mc H$
to affine Hecke algebras with more than one $q$-parameter. That is relevant because such algebras
arise in abundance from reductive $p$-adic groups and types \cite[\S 2.4]{ABPS}. However, up to today
it is unclear how several independent $q$-parameters can be incorporated in a setup with equivariant
$K$-theory or $K$-homology. The situation improves when one formally completes an affine Hecke algebra
with respect to (the kernel of) a central character, as in \cite{Lus-Gr}. Such a completion is
Morita equivalent with a completion of a graded Hecke algebra with respect to a central character.\\

Graded Hecke algebras $\mh H$ with several parameters (now typically called $k$) do admit a geometric
interpretation \cite{Lus-Cusp1,Lus-Cusp2}. (Not all combinations of parameters occur though,
there are conditions on the ratios between the different $k$-parameters.) For this reason graded
Hecke algebras, instead of affine Hecke algebras, play the main role in this paper. 

Such algebras, and minor generalizations called twisted graded Hecke algebras, appear in several
independent ways. Consider a connected reductive group $\mc G$ defined over a non-archimedean
local field $F$. Let $\Rep (\mc G (F))^{\mf s}$ be any Bernstein block in the category of (complex,
smooth) $\mc G (F)$-representations. Locally on the space of characters of the Bernstein centre
of $\mc G (F)$, $\Rep (\mc G (F))^{\mf s}$ is always equivalent with the module category of some 
twisted graded Hecke algebra \cite[\S 7]{SolEnd}. This is derived from an equivalence of 
$\Rep (\mc G (F))^{\mf s}$ with the module category of an algebra which is almost an affine Hecke
algebra, established in full generality in \cite{SolEnd}.

The same kind of algebras arise from enhanced Langlands parameters for $\mc G (F)$ \cite{AMS2}.
That construction involves complex geometry and the cuspidal support map for enhanced L-parameters
from \cite{AMS1}. It matches specific sets of enhanced L-parameters for $\mc G (F)$ with specific
sets of irreducible representations of twisted graded Hecke algebras, see \cite{AMS2,AMS3}.

Like affine Hecke algebras, graded Hecke algebras are related to equivariant sheaves on varieties
associated to complex reductive groups. However, here the sheaves must be constructible and one
uses equivariant cohomology instead of equivariant K-theory. Just equivariant sheaves do not 
suffice to capture the entire structure of graded Hecke algebras, one rather needs differential
complexes of those. Thus we arrive at the (bounded) equivariant derived categories of 
constructible sheaves from \cite{BeLu}. Via intersection cohomology, such objects have many
applications in representation theory, see for instance \cite{Lus-IC,Ach}.\\

\textbf{Main results}\\
Let $G$ be a complex reductive group and let $M$ be a Levi subgroup of $G$. To cover all
enhanced Langlands parameters for $p$-adic groups and all instances of twisted graded Hecke algebras
mentioned above, we must allow disconnected reductive groups. 
Let $q\cE$ be an irreducible $M$-equivariant cuspidal local system on a nilpotent orbit in the Lie 
algebra of $M$. From these data a twisted graded Hecke algebra $\mh H (G,M,q\cE)$ can be constructed 
\cite[\S 4]{AMS2}. As a graded vector space, it is the tensor product of:
\begin{itemize}
\item the algebra $\mc O (\mf t)$ of polynomial functions on Lie$(Z(M^\circ)) = \mf t$, 
with grading 2 times the standard grading,
\item $\C [\mb r]$, where $\mb r$ is a formal variable of degree 2,
\item the (twisted) group algebra of a finite ``Weyl-like" group $W_{q\cE}$ (in degree 0).
\end{itemize}
We will work in $\mc D^b_{G \times \C^\times}(X)$, the $G \times \C^\times$-equivariant bounded 
derived category of constructible sheaves on a complex variety $X$ \cite{BeLu}.
We let $G$ act on its Lie algebra $\mf g$ via the adjoint representation, and we let $\lambda \in 
\C^\times$ act on $\mf g$ as multiplication by $\lambda^{-2}$. 
In \cite{Lus-Cusp1,Lus-Cusp2,AMS2} an important object $K \in \mc D^b_{G \times \C^\times} (\mf g)$ 
was constructed from $q\cE$, by a process that bears some similarity with parabolic induction. 
With $G^\circ$ instead of $G \times \C^\times$, $K$ would be a character sheaf as in \cite{Lus-Char}.
In general it does not fit entirely with Lusztig's notion of character sheaves on disconnected 
reductive groups, because those are only $G^\circ$-equivariant.

Let $\mf g_N$ be the variety of nilpotent elements in the Lie algebra $\mf g$ of $G$ and let $K_N$ 
be the pullback of $K$ to $\mf g_N$. Up to degree shifts, both $K$ and $K_N$ are direct sums of 
simple perverse sheaves. This $K_N$ generalizes the equivariant perverse sheaves used to establish 
the (generalized) Springer correspondence \cite{Lus-Int}. The following was already known for
connected $G$, from \cite{Lus-Cusp1,Lus-Cusp2}, while for disconnected $G$ it follows quickly
from \cite{AMS2}.

\begin{thmintro}\label{thm:A}
(see Theorem \ref{thm:1.2}) \\
There exist natural isomorphisms of graded algebras
\[
\mh H (G,M,q\cE) \longrightarrow \End^*_{\mc D^b_{G \times \C^\times}(\mf g)}(K) \longrightarrow
\End^*_{\mc D^b_{G \times \C^\times}(\mf g_N)}(K_N) .
\]
\end{thmintro}

We point out that here the additional $\C^\times$-action makes things much more interesting 
(like in \cite{KaLu} for K-theory and affine Hecke algebras). Indeed, the simpler version
$\End^*_{\mc D^b_G (\mf g)}(K)$ is isomorphic to the crossed product of $\mc O (\mf t)$ with a 
twisted group algebra of $W_{q\cE}$, and that does not involve any Hecke type relations.

Let $\mc D^b_{G \times \C^\times}(\mf g_N ,K_N)$ be the full triangulated subcategory of
$\mc D^b_{G \times \C^\times}(\mf g_N)$ generated by $K_N$. By analogy with progenerators of
module categories, Theorem \ref{thm:A} indicates that $\mc D^b_{G \times \C^\times}(\mf g_N ,K_N)$ 
should be equivalent to some category of right $\mh H (G,M,q\cE)$-modules. Our geometric objects
are differential complexes of sheaves (up to equivalences), and accordingly we need 
(equivalence classes of complexes of) differential graded $\mh H (G,M,q\cE)$-modules.

\begin{thmintro}\label{thm:B}
(see Theorem \ref{thm:3.3}) \\
There exists an equivalence of categories between 
$\mc D^b_{G \times \C^\times}(\mf g_N ,K_N)$ and\\ $\mc D^b (\mh H (G,M,q\cE)-\Mod_{\mr{fgdg}})$, the 
bounded derived category of finitely generated differential graded right $\mh H (G,M,q\cE)$-modules.
\end{thmintro}

This is a geometric categorification of $\mh H (G,M,q\cE)$, albeit of a different kind than in
\cite{KaLu,Bez}. It is a variation (with $G \times \C^\times$ instead of $G^\circ$) on the derived
version of the generalized Springer correspondence from \cite{Rid,RiRu}. In that setting, the algebra
is $\mc O (\mf t) \rtimes W(G,T)$, which can also be considered as a graded Hecke algebra with 
parameters $k = 0$.
Further, one may regard Theorem \ref{thm:B} as ``formality" of the graded algebra $\mh H (G,M,q\cE)$,
in the following sense. There exists a differential graded algebra $\mc R$ (with nonzero differential)
such that $H^* (\mc R) \cong \mh H (G,M,q\cE)$ and $\mc R$ is formal, that is, quasi-isomorphic with
$H^* (\mc R)$. The equivalence in Theorem \ref{thm:B} maps $\mc D^b_{G \times \C^\times}(\mf g_N ,K_N)$
to $\mc D^b (\mc R -\Mod_{\mr{fgdg}})$ via some Hom-functor, and from there to 
$\mc D^b (\mh H (G,M,q\cE)-\Mod_{\mr{fgdg}})$ by taking cohomology.

From a geometric point of view, it is more natural to consider the entire category 
$\mc D^b_{G \times \C^\times}(\mf g_N)$ in Theorem \ref{thm:B}. It turns out that this category
decomposes, like in a related setting in \cite{RiRu1}:

\begin{thmintro}\label{thm:C}
(see Theorem \ref{thm:3.6}) \\
There exists an orthogonal decomposition 
\[
\mc D^b_{G \times \C^\times}(\mf g_N) = 
\bigoplus\nolimits_{[M,q\cE]_G} \mc D^b_{G \times \C^\times}(\mf g_N ,K_N) .
\vspace{-1mm}
\]
Here $K_N$ is constructed from an $M$-equivariant cuspidal local system $q\cE$ on a nilpotent 
orbit in a Levi Lie-subalgebra $\mr{Lie} (M)$, and the direct sum runs over $G$-conjugacy classes 
of such pairs $(\mr{Lie} (M),q\cE)$.
\end{thmintro}

\textbf{Outlook}\\
Theorems \ref{thm:B} and \ref{thm:C} describe $\mc D^b_{G \times \C^\times}(\mf g_N)$ 
as a derived module category.
Let us point out that the category $\mh H (G,M,q\cE)-\Mod_{\mr{fgdg}}$ in Theorem \ref{thm:B} is
much smaller than the category of ungraded finitely generated right $\mh H (G,M,q\cE)$-modules
(which relates to categories of smooth representations of reductive $p$-adic groups).
In the sequels to this paper \cite{SolSheaves,SolStand} we focus on standard and irreducible 
$\mh H (G,M,q\cE)$-modules, so modules with a fixed central character. Those will be analysed 
using versions of localization for equivariant derived constructible sheaves. In the end, this
will be used to verify the Kazhdan--Lusztig conjecture for $p$-adic groups 
\cite[Conjecture 8.11]{Vog} and the generalized injectivity conjecture \cite{CaSh}. To prove
these conjectures entirely, it is essential to work in the generality of the current paper. 

In words, Theorem \ref{thm:C} says that the $G \times \C^\times$-equivariant derived category of
constructible sheaves on the nilpotent cone $\mf g_N$ decomposes as a direct sum of subcategories
associated to the various involved cuspidal supports $(M,q\cE)$. Let us speculate on how this 
relates to the Langlands program. An element $N \in \mf g_N$ and a semisimple element $g \in G$
with Ad$(g)N = q_F N$ can be used to define a Langlands parameter for a reductive group over a
non-archimedean local field $F$. It may be an unramified L-parameter like in \cite{KaLu}: trivial 
on the inertia group $\mb I_F$ and with $g$ the image of an arithmetic Frobenius element $\Fr$. 
But one can also start with a more complicated L-parameter $\phi$, let $G^\circ$ be the connected 
centralizer of $\phi (\mb I_F)$ in the complex dual group and let $G/G^\circ$ be generated by 
$\phi (\Fr)$.

One may hope that an analogue of Theorem \ref{thm:C} holds for equivariant sheaves on varieties
(or stacks) of Langlands parameters, as defined in \cite{DHKM,Zhu}. It would say that the relevant
sheaves on Langlands parameters can be decomposed as orthogonal direct sums of pieces associated
to suitable cuspidal supports. That would be similar to the Bernstein decomposition of the
category of the smooth complex representations of a reductive $p$-adic group. A result of
this kind is already known for enhanced L-parameters \cite[\S 8]{AMS1}, that covers the cases
of simple equivariant constructible sheaves. To fit with the (conjectural) framework for
geometrization of the local Langlands correspondence from \cite{FaSc,Zhu}, a version of Theorem
\ref{thm:C} with equivariant coherent sheaves on varieties of L-parameters is desirable.\\

\textbf{Structure of the paper}\\
We start with recalling twisted graded Hecke algebras in terms of generators and relations.
We generalize a few results from \cite{SolHecke}, which say that the set of irreducible 
representations of a graded Hecke algebra is essentially independent of the parameters $k$ and $r$. 
Then we prove a generally useful result:

\begin{thmintro}\label{thm:D}
The global dimension of $\mh H (G,M,q\cE)$ equals $\dim (Z(M^\circ)) + 1$.
\end{thmintro}

In Paragraph \ref{par:geomConst} we describe the geometric construction of $\mh H (G,M,q\cE)$ in
detail, and we establish Theorem \ref{thm:A}. Next we check that $K_N$ is a semisimple object
of $\mc D^b_{G \times \C^\times}(\mf g_N)$ and we relate it to parabolic induction for perverse
sheaves -- which is needed for Theorems \ref{thm:B} and \ref{thm:C}. Paragraph \ref{par:centralizer}
is mainly preparation for an argument with localization to $\exp (\C \sigma)$-invariants. We include 
it here because it is closely related to Paragraph \ref{par:geomConst} and because our analysis of 
$(G/P)^\sigma = (G/P)^{\exp (\C \sigma)}$ for $\sigma \in \mf t$ is of independent interest.
Paragraph \ref{par:iso} and Section \ref{sec:cuspidal} (basically the complement of Theorems
\ref{thm:B}, \ref{thm:C} and \ref{thm:D}) will be used directly in \cite{SolSheaves,SolStand}.

Section \ref{sec:description} is dedicated to Theorems \ref{thm:B} and \ref{thm:C}. We prove them by
reduction to the setting of \cite{Rid,RiRu1,RiRu}, where sheaves of $\Ql$-modules on varieties over
fields of positive characteristic are considered. This involves checking many things, among others
that $\mh H (G,M,q\cE)$ is Koszul as differential graded algebra.\\

\textbf{Acknowledgements}\\
We thank Eugen Hellmann for some enlightening conversations.
A big thanks to the referees for their detailed remarks, which helped to avoid several problems
and to substantially clarify the paper.
\vspace{3mm}

\renewcommand{\theequation}{\arabic{section}.\arabic{equation}}
\counterwithin*{equation}{section}

\section{Graded Hecke algebras} 
\label{sec:defGHA}

Let $\mf a$ be a finite dimensional Euclidean space and let $W$ be a finite Coxeter group
acting isometrically on $\mf a$, and hence also on the linear dual space $\mf a^\vee$. 
Let $R \subset \mf a^\vee$ be a reduced integral root system, stable under the action of $W$, 
such that the reflections $s_\alpha$ with $\alpha \in R$ generate $W$. These conditions imply 
that $W$ acts trivially on the orthogonal complement of $\R R$ in $\mf a^\vee$.

Write $\mf t = \mf a \otimes_\R \C$ and let $S(\mf t^\vee) = \mc O (\mf t)$ be the algebra
of polynomial functions on $\mf t$. We also fix a base $\Delta$ of $R$.
Let $\Gamma$ be a finite group which acts faithfully and orthogonally on $\mf a$ and stabilizes 
$R$ and $\Delta$. Then $\Gamma$ normalizes $W$ and $W \rtimes \Gamma$ is a group of 
automorphisms of $(\mf a,R)$. We choose a $W \rtimes \Gamma$-invariant parameter function
$k : R \to \C$. Let ${\mb r}$ be a formal variable, identified with the coordinate function 
on $\C$ (so $\mc O (\C) = \C [\mb r]$).

Let $\natural : \Gamma^2 \to \C^\times$ be a 2-cocycle and inflate it to a 2-cocycle of 
$W \rtimes \Gamma$. Recall that the twisted group algebra $\C[W \rtimes \Gamma,\natural]$ 
has a $\C$-basis $\{ N_w : w \in W \rtimes \Gamma \}$ and multiplication rules
\[
N_w \cdot N_{w'} = \natural (w,w') N_{w w'} . 
\]
In particular it contains the group algebra of $W$.

\begin{prop}\label{prop:1.1}
\textup{\cite[Proposition 2.2]{AMS2}} \\
There exists a unique associative algebra structure
on $\C [W \rtimes \Gamma,\natural] \otimes \mc O (\mf t) \otimes \C[{\mb r}]$ such that:
\begin{itemize}
\item the twisted group algebra $\C[W \rtimes \Gamma,\natural]$ is embedded as subalgebra;
\item the algebra $\mc O (\mf t) \otimes \C[{\mb r}]$ of polynomial functions on $\mf t 
\oplus \C$ is embedded as a subalgebra;
\item $\C[{\mb r}]$ is central;
\item the braid relation $N_{s_\alpha} \xi - {}^{s_\alpha}\xi N_{s_\alpha} = 
k (\alpha) {\mb r} (\xi - {}^{s_\alpha} \xi) / \alpha$ holds for all $\xi \in \mc O (\mf t)$ 
and all simple roots $\alpha$;
\item $N_w \xi N_w^{-1} = {}^w \xi$ for all $\xi \in \mc O (\mf t)$ and $w \in \Gamma$. 
\end{itemize}
\end{prop}
We denote the algebra from Proposition \ref{prop:1.1} by
$\mh H (\mf t, W \rtimes \Gamma,k,\mb{r},\natural)$ and we call it a twisted graded Hecke algebra. 
It is graded by putting $\C[W \rtimes \Gamma,\natural]$ in degree 0 and $\mf t^\vee \setminus \{0\}$ 
and $\mb r$ in degree 2. When $\Gamma$ is trivial, we omit $\natural$ from the notation, and we 
obtain the usual notion of a graded Hecke algebra $\mh H (\mf t, W,k,\mb{r})$. 

Notice that for $k = 0$ Proposition \ref{prop:1.1} yields the crossed product algebra
\begin{equation}\label{eq:1.9}
\mh H (\mf t, W \rtimes \Gamma, 0,\mb{r}, \natural) = \C[\mb{r}] \otimes_\C \mc O (\mf t) \rtimes 
\C [W \rtimes \Gamma, \natural],
\end{equation}
with multiplication rule 
\[
N_w \xi N_w^{-1} = {}^w \xi \qquad w \in W \rtimes \Gamma, \xi \in \mc O (\mf t) .
\]
It is possible to scale all parameters $k(\alpha)$ simultaneously. Namely, scalar 
multiplication with $z \in \C^\times$ defines a bijection $m_z : \mf t^\vee \to \mf t^\vee$,
which clearly extends to an algebra automorphism of $S(\mf t^\vee)$. From Proposition
\ref{prop:1.1} we see that it extends even further, to an algebra isomorphism
\begin{equation}\label{eq:1.16}
m_z : \mh H (\mf t, W \rtimes \Gamma, zk, \mb{r}, \natural) \to \mh H (\mf t,W \rtimes 
\Gamma ,k,\mb{r}, \natural) 
\end{equation}
which is the identity on $\C [W \rtimes \Gamma, \natural] \otimes_\C \C [\mb{r}]$. Notice that for 
$z=0$ the map $m_z$ is well-defined, but no longer bijective. It is the canonical surjection
\[
\mh H (\mf t,W \rtimes \Gamma,0,\mb{r}, \natural) \to 
\C [W \rtimes \Gamma, \natural] \otimes_\C \C [\mb{r}] .
\]
One also encounters versions of $\mh H (\mf t,W \rtimes \Gamma,k,\mb{r},\natural)$ with $\mb{r}$
specialized to a nonzero complex number. In view of \eqref{eq:1.16} it hardly
matters which specialization, so it suffices to look at $\mb r \mapsto 1$. The
resulting algebra $\mh H (\mf t,W \rtimes \Gamma,k, \natural)$ has underlying vector space
$\C [W \rtimes \Gamma, \natural] \otimes_\C \mc O (\mf t)$ and cross relations
\begin{equation}\label{eq:1.17} 
\xi \cdot s_\alpha - s_\alpha \cdot s_\alpha (\xi) = 
k(\alpha) (\xi - s_\alpha (\xi)) / \alpha \qquad \alpha \in \Delta, \xi \in S(\mf t^\vee) .
\end{equation}
Since $\Gamma$ acts faithfully on $(\mf a,\Delta)$, and $W$ acts simply transitively on the
collection of bases of $R$, $W \rtimes \Gamma$ acts faithfully on $\mf a$.
From \eqref{eq:1.17} we see that the centre of $\mh H (\mf t,W \rtimes \Gamma,k, \natural)$ is
\begin{equation}\label{eq:1.18}
Z(\mh H (\mf t,W \rtimes \Gamma,k,\natural)) = S(\mf t^\vee)^{W \rtimes \Gamma} = 
\mc O (\mf t / W \rtimes \Gamma) .
\end{equation}
As a vector space, $\mh H(\mf t,W \rtimes \Gamma,k, \natural)$ is still graded by $\deg (w) = 0$ 
for $w \in W \rtimes \Gamma$ and $\deg (x) = 2$ for $x \in \mf t^\vee \setminus \{0\}$. However, 
it is not a graded algebra any more, because \eqref{eq:1.17} is not homogeneous in the case 
$\xi = \alpha$. Instead, the above grading merely makes $\mh H(\mf t,W \rtimes \Gamma,k, \natural)$
into a filtered algebra. The graded algebra associated to this filtration is obtained by setting 
the right hand side of \eqref{eq:1.17} equal to 0. In other words, the associated graded object 
of $\mh H(\mf t,W \rtimes \Gamma,k, \natural)$ is the crossed product algebra \eqref{eq:1.9}.

Graded Hecke algebras can be decomposed like root systems and reductive Lie algebras.
Let $R_1, \ldots, R_d$ be the irreducible components of $R$. Write
$\mf a_i^\vee = \mr{span}(R_i) \subset \mf a^\vee$, $\mf t_i = \Hom_\R (\mf a_i^\vee,\C)$
and $\mf z = R^\perp \subset \mf t$. Then
\begin{equation}\label{eq:1.7}
\mf t = \mf t_1 \oplus \cdots \oplus \mf t_d \oplus \mf z .
\end{equation}
The inclusions $W(R_i) \to W(R), \mf t_i^\vee \to \mf t^\vee$ and $\mf z^\vee \to \mf t^\vee$
induce an algebra isomorphism
\begin{equation}\label{eq:1.22}
\mh H (\mf t_1, W(R_1), k) \otimes_\C \cdots \otimes_\C \mh H (\mf t_d, W(R_d),k) 
\otimes_\C \mc O (\mf z) \; \longrightarrow \; \mh H (\mf t, W, k) .
\end{equation}
The central subalgebra $\mc O (\mf z) \cong S(\mf z^\vee)$ is of course very simple, so the study of
graded Hecke algebras can be reduced to the case where the root system $R$ is irreducible.

\subsection{Some representation theory} \
\label{par:iso}

We list some isomorphisms of (twisted) graded Hecke algebras that will be useful later on.
For any $z \in \C^\times$, $\mh H (\mf t, W \rtimes \Gamma,k,\mb{r},\natural)$ 
admits a ``scaling by degree" automorphism
\begin{equation}\label{eq:1.11}
x \mapsto z^n x \qquad \text{if } x \in \mh H (\mf t, W \rtimes \Gamma,k,\mb{r},\natural)
\text{ has degree } 2n.
\end{equation}
Extend the sign representation to a character $\sgn$ of $W \rtimes \Gamma$, trivial on $\Gamma$.
That yields the sign involution
\begin{equation}\label{eq:1.6}
\begin{aligned}
& \sgn : \mh H(\mf t,W \rtimes \Gamma,k, \mb r, \natural) \to 
\mh H(\mf t,W \rtimes \Gamma,k, \mb r, \natural) \\
& \sgn (N_w) = \sgn (w) N_w ,\quad \sgn (\mb r) = -\mb r ,\quad \sgn (\xi) = \xi \qquad
w \in W \rtimes \Gamma, \xi \in \mf t^\vee.
\end{aligned}
\end{equation}
Upon specializing $\mb r = 1$, it induces an algebra isomorphism
\[
\sgn : \mh H(\mf t,W \rtimes \Gamma,k, \natural) \to \mh H(\mf t,W \rtimes \Gamma,-k, \natural) .
\]
More generally, we can pick a sign $\epsilon (s_\alpha)$ for every simple reflection $s_\alpha \in W$,
such that $\epsilon (s_\alpha) = \epsilon (s_\beta)$ if $s_\alpha$ and $s_\beta$ are conjugate
in $W \rtimes \Gamma$. Then $\epsilon$ extends uniquely to a character of $W \rtimes \Gamma$ 
trivial on $\Gamma$ (and every character of $W \rtimes \Gamma$ which is trivial on $\Gamma$ has
this form). Define a new parameter function $\epsilon k$ by 
\[
\epsilon k (\alpha) = \epsilon (s_\alpha) k(\alpha) .
\]
Then there are algebra isomorphisms
\begin{equation}
\begin{array}{llll}
\phi_\epsilon : & \mh H(\mf t,W \rtimes \Gamma,k,\mb r,\natural) & \to & 
\mh H(\mf t,W \rtimes \Gamma,\epsilon k,\mb r,\natural) ,\\ 
\phi_\epsilon : & \mh H(\mf t,W \rtimes \Gamma,k,\natural) & \to & 
\mh H(\mf t,W \rtimes \Gamma,\epsilon k,\natural) ,\\
\multicolumn{4}{l}{\phi_\epsilon (N_w) = \epsilon (w) N_w ,\quad \phi_\epsilon (\mb r) = \mb r ,\quad
\phi_\epsilon (\xi) = \xi, \qquad w \in W \rtimes \Gamma, \xi \in \mc O (\mf t) .}
\end{array}
\end{equation}
Notice that for $\epsilon$ equal to the sign character of $W$, $\phi_\epsilon$ agrees with $\sgn$ 
from \eqref{eq:1.6} on $\mh H(\mf t,W \rtimes \Gamma,k,\natural)$ but not on 
$\mh H(\mf t,W \rtimes \Gamma,k,\mb r,\natural)$.

For $R$ irreducible of type $B_n, C_n, F_4$ or $G_2$, there are two further nontrivial
possible $\epsilon$'s. Consider the characters $\epsilon_s, \epsilon_l$ of $W$ with
\[
\epsilon_s (s_\alpha) = \left\{ \begin{array}{cc}
1 & \alpha \text{ long} \\
-1 & \alpha \text{ short}
\end{array}\right. ,\qquad
\epsilon_l (s_\alpha) = \left\{ \begin{array}{cc}
1 & \alpha \text{ short} \\
-1 & \alpha \text{ long}
\end{array}\right. .
\]
Since $\Gamma$ acts isometrically on $\mf a$, $\epsilon_l$ and $\epsilon_s$ are $\Gamma$-invariant.
Thus we obtain algebra isomorphisms
\[
\phi_{\epsilon_s} : \mh H (\mf t,W \rtimes \Gamma,k,\natural) \to 
\mh H(\mf t,W \rtimes \Gamma,\epsilon_s k,\natural) ,\;
\phi_{\epsilon_l} : \mh H (\mf t,W \rtimes \Gamma,k,\natural) \to 
\mh H(\mf t,W \rtimes \Gamma,\epsilon_l k,\natural) .
\]

\begin{lem}\label{lem:1.4}
Let $\mh H(\mf t,W \rtimes \Gamma,k,\natural)$ be a twisted graded Hecke algebra with a
real-valued parameter function $k$. Then it is isomorphic to a twisted graded Hecke algebra
$\mh H(\mf t,W \rtimes \Gamma,\epsilon k,\natural)$ with $\epsilon k : R \to \R_{\geq 0}$,
via an isomorphism $\phi_\epsilon$ that is the identity on $\mc O (\mf t \oplus \C)$.
\end{lem}
\begin{proof}
Define
\[
\epsilon (s_\alpha) = \left\{ \begin{array}{cc}
1 & k (\alpha) \geq 0 \\
-1 & k (\alpha) < 0 
\end{array}\right. .
\]
Since $k$ is $\Gamma$-invariant, this extends to a $\Gamma$-invariant quadratic character of $W$.
Then $\phi_\epsilon$ has the required properties.
\end{proof}

With the above isomorphisms we will generalize the results of \cite[\S 6.2]{SolHecke}, from
graded Hecke algebras with positive parameters to twisted graded Hecke algebras with real parameters.

For the moment, we let $\mh H$ stand for either $\mh H (\mf t,W \rtimes \Gamma,k,\mb r,\natural)$ or
$\mh H (\mf t,W \rtimes \Gamma,k,\natural)$. Every finite dimensional $\mh H$-module $V$ is the direct 
sum of its generalized $\mc O (\mf t)$-weight spaces
\[
V_\lambda := \{ v \in V : (\xi - \xi(\lambda))^{\dim V} v = 0 \; \forall \xi \in \mc O (\mf t) \} 
\qquad \lambda \in \mf t . 
\]
We denote the set of $\mc O (\mf t)$-weights of $V$ by
\[
\mr{Wt}(V) = \{ \lambda \in \mf t : V_\lambda \neq 0 \} .
\]
Let $\mf a^{-}$ be the obtuse negative cone in $\R R \subset \mf a$ determined by $(R,\Delta)$. 
We denote the interior of $\mf a^{-}$ in $\R R$ by $\mf a^{--}$. 
We recall that a finite dimensional $\mh H$-module $V$ is tempered if 
\[
\mr{Wt}(V) \subset \mf a^{-} \oplus i \mf a
\] 
and that $V$ is essentially discrete series if, with $\mf z$ as in \eqref{eq:1.7}:
\[
\mr{Wt}(V) \subset \mf a^{--} \oplus (\mf z \cap \mf a) \oplus i \mf a. 
\]
For a subset $U$ of $\mf t$ we let $\Modf{U}(\mh H)$ be the category of finite dimensional
$\mh H$-modules $V$ with Wt$(V) \subset U$. For example, we have the category of $\mh H$-modules
with ``real" weights $\Modf{\mf a}(\mh H)$.
We indicate a subcategory/subset of tempered modules by a subscript ``temp". In particular, we have
the category of finite dimensional tempered $\mh H$-modules $\Mod_{\mr{fl}}(\mh H )_{\mr{temp}}$.

We want to compare the irreducible representations of
\[
\mh H (\mf t,W \rtimes \Gamma,k,\natural) = \mh H (\mf t,W \rtimes \Gamma,k,\mb r,\natural) / (\mb r - 1) 
\]
with those of
\[
\mh H (\mf t,W \rtimes \Gamma,0,\natural) = \mh H (\mf t,W \rtimes \Gamma,k,\mb r,\natural) / (\mb r) . 
\]
The latter algebra has $\Irr (\C [W \rtimes \Gamma,\natural])$ as the set of irreducible representations
on which $\mc O (\mf t)$ acts via evaluation at $0 \in \mf t$. The correct analogue of this for
$\mh H (\mf t,W \rtimes \Gamma,k,\natural)$, at least with $k$ real-valued, is
\[
\Irr_{\mf a} (\mh H (\mf t,W \rtimes \Gamma,k,\natural))_{\mr{temp}} 
:= \Irr (\mh H (\mf t,W \rtimes \Gamma,k, \natural))_{\mr{temp}} \cap 
\Modf{\mf a}( \mh H (\mf t,W \rtimes \Gamma,k,\natural)) .
\]
As $\C [W \rtimes \Gamma, \natural]$ is a subalgebra of $\mh H (\mf t,W \rtimes \Gamma,k,\natural)$,
there is a natural restriction map
\[
\Res_{W \rtimes \Gamma} : \Mod_{\mr{fl}} (\mh H (\mf t,W \rtimes \Gamma,k,\natural)) \to
\Mod_{\mr{fl}}(\C [W \rtimes \Gamma, \natural]) .
\]
However, when $k \neq 0$ this map usually does not preserve irreducibility, not even on
$\Irr_{\mf a} (\mh H (\mf t,W \rtimes \Gamma,k,\natural))_{\mr{temp}}$.
 
In the remainder of this paragraph we assume that the parameter function $k$ only takes real values.
Let $\epsilon$ be as in Lemma \ref{lem:1.4}. Since $\phi_\epsilon$ is the identity on 
$\mc O (\mf t \oplus \C)$, it induces equivalences of categories
\[
\begin{array}{llll}
\Modf{U}(\mh H (\mf t,W \rtimes \Gamma,\epsilon k,\natural) ) & \longrightarrow &
\Modf{U} (\mh H (\mf t,W \rtimes \Gamma,k,\natural)) & U \subset \mf t , \\
\Mod_{\mr{fl}}(\mh H (\mf t,W \rtimes \Gamma,\epsilon k,\natural) )_{\mr{temp}} & \longrightarrow &
\Mod_{\mr{fl}} (\mh H (\mf t,W \rtimes \Gamma,k,\natural))_{\mr{temp}} &
\end{array}
\]
and a bijection
\[
\Irr_{\mf a} (\mh H (\mf t,W \rtimes \Gamma,\epsilon k,\natural) )_{\mr{temp}} \longrightarrow
\Irr_{\mf a} (\mh H (\mf t,W \rtimes \Gamma,k,\natural) )_{\mr{temp}} .
\]

\begin{thm}\label{thm:1.3}
Let $k : R \to \R$ be a $\Gamma$-invariant parameter function.
\enuma{
\item The set $\Res_{W \rtimes \Gamma} (\Irr_{\mf a} \mh H (\mf t,W \rtimes \Gamma,k,
\natural)_{\mr{temp}} )$ is a $\Z$-basis of $\Z \, \Irr (\C [W \rtimes \Gamma, \natural])$.
}
Suppose that the restriction of $k$ to any type $F_4$ component of $R$ has $k(\alpha) = 0$ for a root 
$\alpha$ in that component or is the form $\epsilon k'$ for a character $\epsilon : W (F_4) \to \{\pm 1\}$
and a parameter function $k' : F_4 \to \R_{>0}$ which is geometric in the sense of the next remark.
\enuma{ \setcounter{enumi}{1}
\item There exist total orders on $\Irr_{\mf a} (\mh H (\mf t,W \rtimes \Gamma,k,\natural)_{\mr{temp}})$ 
and on $\Irr (\C [W \rtimes \Gamma ,\natural])$, such that the matrix of the $\Z$-linear map
\[
\Res_{W \rtimes \Gamma} : \Z \, \Irr_{\mf a} (\mh H (\mf t,W \rtimes \Gamma,k,\natural))_{\mr{temp}} 
\to \Z \, \Irr (\C [W \rtimes \Gamma ,\natural])
\]
is upper triangular and unipotent.
\item There exists a unique bijection
\[
\zeta_{\mh H (\mf t,W \rtimes \Gamma,k,\natural))} : \Irr_{\mf a} (\mh H (\mf t,W \rtimes \Gamma,
k,\natural))_{\mr{temp}} \to \Irr (\C [W \rtimes \Gamma, \natural])
\]
such that $\zeta_{\mh H (\mf t,W \rtimes \Gamma,\epsilon k,\natural)} (\pi)$ always
occurs in $\Res_{W \rtimes \Gamma} (\pi)$.
}
\end{thm}
\begin{rem} 
Geometric parameter functions will appear in Section \ref{sec:cuspidal}. Let us make the allowed 
parameter functions for a type $F_4$ root system explicit here. Write $k = (k(\alpha), 
k(\beta))$ where $\alpha$ is short root and $\beta$ is a long root. The possibilities are
\[
(0,0), (c,0), (0,c), (c,c), (2c,c), (c/2,c), (4c,c), (-c,c), (-2c,c), (-c/2,c), (-4c,c) ,
\]
where $c \in \R^\times$ is arbitrary. We expect that Theorem \ref{thm:1.3} also holds without extra
conditions for type $F_4$.
\end{rem}
\begin{proof}
(a) is known from \cite[Proposition 1.7]{SolK}. The proof of that shows we can reduce the entire
theorem to the case where $\natural$ is trivial. We assume that from now on, and omit $\natural$
from the notations.

Parts (b) and (c) were shown in \cite[Theorem 6.2]{SolHecke}, provided that $k(\alpha) \geq 0$
for all $\alpha \in R$. Choose $\epsilon$ as in Lemma \ref{lem:1.4}, so that 
$\epsilon k : R \to \R_{\geq 0}$. For $V \in \Mod_{\mr{fl}}(\mh H (\mf t, W ,\epsilon k))$ we have
\[
\Res_{W} (\phi_\epsilon^* V) = \Res_{W} (V) \otimes \epsilon ,
\]
so we obtain a commutative diagram
\begin{equation}\label{eq:1.8}
\begin{array}{ccc}
\Z \, \Irr_{\mf a} (\mh H (\mf t, W ,\epsilon k) )_{\mr{temp}} & \xrightarrow{\Res_W} & \Z \, \Irr (W) \\
\downarrow \phi_\epsilon^* & & \downarrow \otimes \epsilon \\
\Z \, \Irr_{\mf a} (\mh H (\mf t, W , k) )_{\mr{temp}} & \xrightarrow{\Res_W} & \Z \, \Irr (W) 
\end{array}
\end{equation}
All the maps in this diagram are bijective and the vertical maps preserve irredu\-ci\-bi\-li\-ty.
Thus the theorem for $\mh H (\mf t,W,\epsilon k)$ implies it for $\mh H (\mf t,W,k)$.

The commutative diagram \eqref{eq:1.8} also allows us to extend \cite[Lemma 6.5]{SolHecke} from
$\mh H (\mf t,W, \epsilon k)$ to $\mh H (\mf t,W, k)$. Then we can finish our proof for
$\mh H (\mf t,W \rtimes \Gamma, k)$ by applying \cite[Lemma 6.6]{SolHecke}.
\end{proof}

\begin{thm}\label{thm:1.10}
Let $\mh H (\mf t,W \rtimes \Gamma,k,\natural)$ be as in Theorem \ref{thm:1.3}.b.
There exists a canonical bijection
\[
\zeta_{\mh H (\mf t,W \rtimes \Gamma,k,\natural))} : \Irr (\mh H (\mf t,W \rtimes \Gamma,
k,\natural)) \to \Irr (\mh H (\mf t,W \rtimes \Gamma,0,\natural)) 
\]
which (as well as its inverse)
\begin{itemize}
\item respects temperedness,
\item preserves the intersections with $\Modf{\mf a}$,
\item generalizes Theorem \ref{thm:1.3}.c, via the identification 
\[
\Irr_{\mf a} (\mh H (\mf t,W \rtimes \Gamma,0,\natural))_{\mr{temp}} = 
\Irr (\C [W \rtimes \Gamma,\natural]) .
\]
\end{itemize}
\end{thm}
\begin{proof}
As discussed in the proof of Theorem \ref{thm:1.3}.a, we can easily reduce to the case where
$\natural$ is trivial. In \cite[Proposition 6.8]{SolHecke}, that case is derived from 
\cite[Theorem 6.2]{SolHecke} (under more strict conditions on the parameters $k$). Using Theorem
\ref{thm:1.3} instead of \cite[Theorem 6.2]{SolHecke}, this works for all parameters allowed
in Theorem \ref{thm:1.3}. Although \cite[Proposition 6.8]{SolHecke} is only formulated for
irreducible representations in $\Modf{\mf a} (\mh H (\mf t,W \rtimes \Gamma,k))$,
the argument applies to all of $\Irr (\mh H (\mf t,W \rtimes \Gamma,k))$. 
\end{proof}

\subsection{Global dimension} \
\label{par:gldim}

We want to determine the global dimension of $\mh H (\mf t, W \rtimes \Gamma, k, \mb r, \natural)$. 
For $\mh H (\mf t,W,kr)$ it has already been done in \cite{SolGHA}, and our argument is based on
reduction to that case. A lower bound for the global dimension is easily obtained:

\begin{lem}\label{lem:1.13}
$\mr{gl. \, dim} (\mh H (\mf t, W \rtimes \Gamma, k, \mb r, \natural)) \geq \dim_\C (\mf t \oplus \C)$.
\end{lem}
\begin{proof}
We abbreviate $\mh H = \mh H (\mf t, W \rtimes \Gamma, k, \mb r, \natural)$.
Pick $\lambda \in \mf t$ such that $w \lambda \neq \lambda$ for all $w \in W \rtimes \Gamma 
\setminus \{1\}$. Fix any $r \in \C$ and let $\C_{\lambda,r}$ be the onedimensional 
$\mc O (\mf t \oplus \C)$-module with character $(\lambda,r)$. By \cite[Theorem 6.4]{BaMo},
which generalizes readily to include $\Gamma$, the $\mc O (\mf t)$-weights of
\begin{equation}\label{eq:1.35}
\Res^{\mh H}_{\mc O (\mf t \oplus \C)} \mr{ind}_{\mc O (\mf t \oplus \C)}^{\mh H} \C_{\lambda,r} 
\end{equation}
are precisely the $w \lambda$ with $w \in W \rtimes \Gamma$. These are all different and the
dimension of \eqref{eq:1.35} is $|W \rtimes \Gamma|$, so \eqref{eq:1.35} must be isomorphic with
$\bigoplus\nolimits_{w \in W \rtimes \Gamma} \C_{w \lambda,r}$.
By Frobenius reciprocity
\begin{align}\nonumber  
\mr{Ext}^n_{\mh H} \big( \mr{ind}_{\mc O (\mf t \oplus \C)}^{\mh H} \C_{\lambda,r} ,
\mr{ind}_{\mc O (\mf t \oplus \C)}^{\mh H} \C_{\lambda,r} \big) \cong
\mr{Ext}^n_{\mc O (\mf t \oplus \C)} \big( \C_{\lambda,r}, \Res^{\mh H}_{\mc O (\mf t \oplus \C)} 
\mr{ind}_{\mc O (\mf t \oplus \C)}^{\mh H} \C_{\lambda,r} \big) \\
\label{eq:1.65} \cong \bigoplus\nolimits_{w \in W \rtimes \Gamma} \mr{Ext}^n_{\mc O (\mf t \oplus \C)} 
\big( \C_{\lambda,r},\C_{w \lambda, r} \big) = \mr{Ext}^n_{\mc O (\mf t \oplus \C)} 
\big( \C_{\lambda,r}, \C_{\lambda, r} \big) .
\end{align}
It is well-known (and can be computed with a Koszul resolution) that the last expression equals
(with $\mf T$ for tangent space)
\begin{equation}\label{eq:1.66}
\bigwedge\nolimits^n \big( \mf T_{\lambda,r} (\mf t \oplus \C) \big) = 
\bigwedge\nolimits^n (\mf t \oplus \C) .
\end{equation}
This is nonzero when $0 \leq n \leq \dim_\C (\mf t \oplus \C)$, so the global dimension
must be at least $\dim_\C (\mf t \oplus \C)$.
\end{proof}

With a general argument, the computation of the global dimension of \\ $\mh H (\mf t, W \rtimes 
\Gamma, k, \mb r, \natural)$ can be reduced to the cases with $\Gamma = \{1\}$.

\begin{lem}\label{lem:1.11}
Let $\Gamma$ be a finite group acting by automorphisms on a complex algebra $A$. Let
$\natural : \Gamma^2 \to \C^\times$ be a 2-cocycle and build the twisted crossed product
$A \rtimes \C [\Gamma,\natural]$ with multiplication relations as in Proposition \ref{prop:1.1} --
the role of $\mh H (\mf t,W,k,\mb r)$ is played by $A$. Then
\[
\mr{gl. \, dim} (A \rtimes \C [\Gamma,\natural]) = \mr{gl. \, dim} (A) .
\]
\end{lem}
\begin{proof}
For any $A$-module $M$
\[
\mr{Res}_A^{A \rtimes \C [\Gamma,\natural]} \mr{ind}_A^{A\rtimes \C [\Gamma,\natural]} M
\cong \bigoplus\nolimits_{\gamma \in \Gamma} \gamma^* (M) .
\]
Hence $\mr{Ext}^n_A (M',M)$ is a direct summand of 
\[
\mr{Ext}^n_A \big( M', \mr{Res}_A^{A \rtimes \C [\Gamma,\natural]} \mr{ind}_A^{A\rtimes 
\C [\Gamma,\natural]} M \big) \cong
\mr{Ext}^n_A \big( \mr{ind}_A^{A\rtimes \C [\Gamma,\natural]} M',\mr{ind}_A^{A\rtimes 
\C [\Gamma,\natural]} M \big) .
\]
In particular gl. dim $(A) \leq$ gl. dim $(A\rtimes \C [\Gamma,\natural])$.

For any $A \rtimes \C [\Gamma,\natural]$-module $V$ there is a surjective module
homomorphism 
\[
\begin{array}{cccc}
p : & \mr{ind}_A^{A\rtimes \C [\Gamma,\natural]} \mr{Res}_A^{A \rtimes \C [\Gamma,\natural]} V &
\to & V \\
& x \otimes v & \mapsto & x v 
\end{array} .
\]
On the other hand, there is a natural injection
\[
\begin{array}{cccc}
\imath : & V & \to &
\mr{ind}_A^{A\rtimes \C [\Gamma,\natural]} \mr{Res}_A^{A \rtimes \C [\Gamma,\natural]} V \\
& v & \mapsto & \sum_{\gamma \in \Gamma} N_\gamma^{-1} \otimes N_\gamma v  
\end{array} .
\]
This in fact a module homomorphism. Namely, for $a \in A$:
\begin{align*}
\imath (a v) & = \sum\nolimits_{\gamma \in \Gamma} N_\gamma^{-1} \otimes N_\gamma a v =
\sum\nolimits_{\gamma \in \Gamma} N_\gamma^{-1} \otimes \gamma (a) N_\gamma v \\
& = \sum\nolimits_{\gamma \in \Gamma} N_\gamma^{-1} \gamma (a) \otimes N_\gamma v =
\sum\nolimits_{\gamma \in \Gamma} a N_\gamma^{-1} \otimes N_\gamma v = a \imath (v) . 
\end{align*}
Similarly, for $g \in \Gamma$:
\begin{align*}
\imath (N_g v) & = \sum\nolimits_{\gamma \in \Gamma} N_\gamma^{-1} \otimes N_\gamma N_g v =
\sum\nolimits_{\gamma \in \Gamma} N_g N_g^{-1} N_\gamma^{-1} \otimes N_\gamma N_g v \\
& = \sum\nolimits_{\gamma \in \Gamma} N_g N_{\gamma g}^{-1} \otimes \ind_{\mc O 
(\mf t \oplus \C)}^{\mh H} \C_{\lambda,r}s N_{\gamma g} v =
\sum\nolimits_{h \in \Gamma} N_g N_h^{-1} \otimes N_h v = N_g \imath (v) . 
\end{align*}
Clearly $p \circ \imath = |\Gamma| \, \mr{id}_V$, so
\[
\mr{ind}_A^{A\rtimes \C [\Gamma,\natural]} \mr{Res}_A^{A \rtimes \C [\Gamma,\natural]} V \cong
V \oplus \ker p \qquad \text{as } A \rtimes \C [\Gamma,\natural] \text{-modules.}
\]
For any second $A \rtimes \C [\Gamma,\natural]$-module $V'$,
$\mr{Ext}^n_{A \rtimes \C [\Gamma,\natural]} (V,V')$ is a direct summand of
\[
\mr{Ext}^n_{A \rtimes \C [\Gamma,\natural]} \big( V \oplus \ker p,V' \big) =
\mr{Ext}^n_{A \rtimes \C [\Gamma,\natural]} \big( \mr{ind}_A^{A\rtimes \C [\Gamma,\natural]} 
\mr{Res}_A^{A \rtimes \C [\Gamma,\natural]} V,V' \big) .
\]
By Frobenius reciprocity the latter is isomorphic with 
\[
\mr{Ext}^n_A \big( \mr{Res}_A^{A \rtimes \C [\Gamma,\natural]} V, 
\mr{Res}_A^{A \rtimes \C [\Gamma,\natural]} V' \big). 
\]
Hence $\mr{Ext}^n_{A \rtimes \C [\Gamma,\natural]} (V,V')$ vanishes whenever $n >\,$gl. dim$(A)$.
\end{proof}

In view of Lemma \ref{lem:1.11} and the construction of $\mh H (\mf t, W \rtimes \Gamma,k,
\mb{r},\natural)$, it can be expected that it has the same global dimension as
$\mc O (\mf t \oplus \C)$. The latter equals $\dim_\C (\mf t \oplus \C)$,
see for instance \cite[Theorem 4.3.7]{Wei}. 

While the global dimensions of these algebras do indeed agree, Lemma \ref{lem:1.11} does not 
suffice to show that. One complication is that a map like $\imath$ above is not a module 
homomorphism in the setting of the group $W \rtimes \Gamma$ and the algebra 
$\mc O (\mf t \oplus \C)$, when the parameters of the Hecke algebra are nonzero. 

The centre of $\mh H (\mf t,W,k,\mb r)$ was identified in \cite[Proposition 4.5]{Lus-Gr} as
\begin{equation}\label{eq:1.56}
Z(\mh H (\mf t,W,k,\mb r)) = \mc O (\mf t)^W \otimes \C [\mb r] .
\end{equation}
From this and Proposition \ref{prop:1.1} we see that 
\begin{equation}\label{eq:1.61}
\mh H (\mf t,W,k,\mb r) \text{ has finite rank as module over } Z(\mh H (\mf t,W,k,\mb r)). 
\end{equation}
For $r \in \C$, let $\hat{\mh H}_r$ be the formal completion of $\mh H (\mf t,W,k,\mb r)$ with 
respect to the ideal $(\mb r - r)$ of $\C [\mb r]$. By \eqref{eq:1.56}, $\hat{\mh H}_r$ is also 
the formal completion of $\mh H (\mf t,W,k,\mb r)$ with respect to the ideal 
\[
I_r = Z(\mh H (\mf t,W,k,\mb r)) (\mb r - r) \; \subset \; Z(\mh H (\mf t,W,k,\mb r)). 
\]
For an $\mh H (\mf t,W,k,\mb r)$-module $V$, we denote its completion with respect to 
$(\mb r -r)$, or equivalently with respect to $I_r$, by $\hat V_r$. If $V$ is finitely 
generated as $\mh H (\mf t,W,k,\mb r)$-module, then by \eqref{eq:1.61} it is also finitely 
generated over $Z(\mh H (\mf t,W,k,\mb r))$, so the natural map 
\begin{equation}\label{eq:1.58}
Z(\hat{\mh H}_r) \otimes_{Z(\mh H(\mf t,W,k,\mb r))} V =
\hat{\mh H}_r \otimes_{\mh H (\mf t,W,k,\mb r)} V 
\longrightarrow \varprojlim_n V / I_r^n V = \hat V_r
\end{equation}
is an isomorphism of $\hat{\mh H}_r$-modules.

\begin{lem}\label{lem:1.16}
$\mr{gl. \, dim} (\mh H (\mf t, W \rtimes \Gamma, k, \mb r, \natural )) \leq 
\mr{sup}_{r \in \C} \, \mr{gl. \, dim}(\hat{\mh H}_r).$
\end{lem}
\begin{proof}
By Lemma \ref{lem:1.11} we may assume that $\Gamma = \{1\}$, so that $\natural$ disappears.
We abbreviate $\mh H = \mh H (\mf t,W,k,\mb r)$. All the algebras in this proof are 
Noetherian, so by \cite[Proposition 4.1.5]{Wei} their global dimensions equal their
Tor-dimensions. We will use both the characterization in terms of Ext-groups and that
in terms of Tor-groups, whatever we find more convenient.

Let $V$ and $M$ be finitely generated $\mh H$-modules. By \eqref{eq:1.58}, exactness of
completion for finitely generated $Z(\mh H)$-modules and \cite[Corollary 3.2.10]{Wei}:
\begin{equation}\label{eq:1.30} 
Z(\hat{\mh H}_r) \otimes_{Z(\mh H)} \mr{Tor}_m^{\mh H} (V,M)  \cong
\mr{Tor}_m^{\hat{\mh H}_r} (\hat{V}_r, \hat{M}_r) .
\end{equation}
It is known, e.g. from \cite[Lemma 3]{KNS}, that $V$ and $M$ have projective resolutions 
consisting of free $Z(\mh H)$-modules of finite rank.
It follows that $\mr{Tor}_m^{\mh H} (V,M)$ is finitely generated as $Z (\mh H)$-module.
By \cite[Lemma 2.9]{SolTwist}, $\mr{Tor}_m^{\mh H} (V,M)$ is nonzero if and only if 
its formal completion with respect to some maximal ideal $I$ of $Z(\mh H)$ is nonzero. That 
happens if and only if \eqref{eq:1.30} is nonzero for the $r \in \C$ with $\mb r - r \in I$. 

For any $m \leq \mr{gl. \, dim}(\mh H)$ such $V,M$ exist, because finitely generated modules
suffice to detect the global dimension of a Noetherian ring \cite[\S 4.1]{Wei}. 
It follows that $\mr{gl. \, dim}(\hat{\mh H}_r) \geq m$ for the appropriate $r$.
\end{proof}

It remains to find a good upper bound for the global dimension of $\hat{\mh H}_r$.

\begin{thm}\label{thm:1.12}
\enuma{
\item For any $r \in \C$, the global dimension of $\hat{\mh H}_r$ equals $\dim_\C (\mf t) + 1$.
\item The global dimension of $\mh H (\mf t, W \rtimes \Gamma, k, 
\mb r, \natural)$ equals $\dim_\C (\mf t) + 1$.
}
\end{thm}
\begin{proof}(a)
By Lemma \ref{lem:1.11} it suffices to consider the cases with $\Gamma = \{1\}$. The crucial 
point of our proof is that the global dimension of the graded Hecke algebra 
\[
\mh H / (\mb r - r) = \mh H (\mf t,W,rk)
\] 
has already been computed, and equals $\dim_\C (\mf t)$ \cite[Theorem 5.3]{SolGHA}. 
For any $\mh H /(\mb r - r)$-module $V_1$, \cite[Theorem 4.3.1]{Wei} provides a comparison of 
projective dimensions
\begin{equation}\label{eq:1.67}
\mr{pd}_{\mh H}(V_1) \leq \mr{pd}_{\mh H / (\mb r - r)}(V_1) + \mr{pd}_{\mh H}(\mh H / (\mb r - r)) .
\end{equation}
From the short exact sequence 
\[
0 \to \mh H \xrightarrow{\mb r - r} \mh H \to \mh H / (\mb r - r) \to 0
\]
we see that $\mr{pd}_{\mh H}(\mh H / (\mb r - r)) = 1$. Hence
\begin{equation}\label{eq:1.32}
\mr{pd}_{\mh H}(V_1) \leq \mr{pd}_{\mh H / (\mb r - r)}(V_1) + 1 \leq \dim_\C (\mf t) + 1.
\end{equation}
In other words, $\mr{Tor}_m^{\mh H}(V_1,M) = 0$ for all $m > \dim_\C (\mf t) + 1$.

Let $V_2$ be an $\mh H$-module on which $(\mb r - r)^2$ acts as 0. In the short exact sequence
\[
0 \to (\mb r - r) V_2 \to V_2 \to V_2 / (\mb r - r) V_2 \to 0 ,
\]
$\mb r - r$ annihilates both the outer terms, so \eqref{eq:1.32} applies to them. Applying
$\mr{Tor}_*^{\mh H}(?,M)$ to this short exact sequence yields a long exact sequence, and taking
\eqref{eq:1.32} into account we see that $\mr{Tor}_m^{\mh H}(V_2,M) = 0$ for all 
$m > \dim_\C (\mf t) + 1$. 

This argument can be applied recursively, and then it shows that
\begin{equation}\label{eq:1.33}
\mr{Tor}_m^{\mh H}(V_n,M) = 0 \qquad \text{if } m > \dim_\C (\mf t) + 1 \text{ and }
(\mb r - r)^n V_n = 0 \text{ for some } n \in \N .
\end{equation}
Assume now that $V$ and $M$ are finitely generated $\hat{\mh H}_r$-modules. By \eqref{eq:1.61}
they are also finitely generated over $Z(\hat{\mh H}_r) = \widehat{Z(\mh H)}_r$, and therefore
\begin{equation}\label{eq:1.60}
V \cong \varprojlim_n V / I_r^n V \cong \varprojlim_n V / (\mb r - r)^n V .
\end{equation}
By \eqref{eq:1.30} and \eqref{eq:1.33} 
\begin{equation}\label{eq:1.34}
\mr{Tor}_m^{\hat{\mh H}_r} (V / (\mb r - r)^n V, M) = 0 \qquad \text{for }  
m > \dim_\C (\mf t) + 1 \text{ and } n \in \N .
\end{equation}
Let $P_* \to M$ be a resolution by free $\hat{\mh H}_r$-modules $P_i$ of finite rank $\mu_i$
(this is possible because $\hat{\mh H}_r$ is Noetherian). Then
\[
\mr{Tor}_m^{\hat{\mh H}_r} (V / (\mb r - r)^n V, M) = 
H_m \big( V / (\mb r - r)^n V \otimes_{\hat{\mh H}_r} \hat{\mh H}_r^{\mu_*}, d_* 
\big) = H_m \big( (V / (\mb r - r)^n V )^{\mu_*} , d_* \big) .
\]
Here the sequence of differential complexes $\big( (V / (\mb r - r)^n V )^{\mu_*} , d_* \big)$,
indexed by $n \in \N$, satisfies the Mittag--Leffler condition because the transition maps are 
surjective. By \eqref{eq:1.60} the inverse limit of the sequence is $(V^{\mu_*}, d_*)$, which 
computes $\mr{Tor}_m^{\hat{\mh H}_r} (V,M)$.
According to \cite[Theorem 3.5.8]{Wei} there is a short exact sequence
\[
0 \to \varprojlim_n \! {}^1 \mr{Tor}_{m+1}^{\hat{\mh H}_r} (V / (\mb r - r)^n V, M) \to
\mr{Tor}_m^{\hat{\mh H}_r} (V,M) \to 
\varprojlim_n \mr{Tor}_m^{\hat{\mh H}_r} (V / (\mb r - r)^n V, M) \to 0.
\]
For $m > \dim_\C (\mf t) + 1$, \eqref{eq:1.34} shows that both outer terms vanish, so 
$\mr{Tor}_m^{\hat{\mh H}_r} (V,M) = 0$ as well. Hence 
\begin{equation}\label{eq:1.59}
\text{gl. dim}(\hat{\mh H}_r) \leq \dim_\C (\mf t) + 1 ,
\end{equation}
which already suffices for part (b).
Consider the $\hat{\mh H}_r$-module 
\[
V_{\lambda,r} := \ind_{\mc O (\mf t \oplus \C)}^{\mh H} \C_{\lambda,r} =
\ind_{\widehat{\mc O (\mf t \oplus \C)}_r}^{\hat{\mh H}_r} \C_{\lambda,r}
\]
from \eqref{eq:1.35}. Analogous to \eqref{eq:1.65} and \eqref{eq:1.66}, one computes that
\[
\mr{Ext}^n_{\hat{\mh H}_r}(V_{\lambda,r},V_{\lambda,r}) \neq 0
\qquad \text{for } 0 \leq n \leq \dim_\C (\mf t) + 1 ,
\]  
which shows that \eqref{eq:1.59} is actually an equality.\\
(b) This follows from Lemmas \ref{lem:1.13} and \ref{lem:1.16} in combination with \eqref{eq:1.59}.
\end{proof}

\section{Equivariant sheaves and equivariant cohomology} 
\label{sec:cuspidal}

We follow the setup from \cite{Lus-Cusp1,Lus-Cusp2,AMS1,AMS2}. In these references a graded Hecke
algebra was associated to a cuspidal local system on a nilpotent orbit for a complex reductive
group, via equivariant cohomology. For future applications to Langlands parameters we deal not 
only with connected groups, but also with disconnected reductive groups $G$.

We work in the $G$-equivariant bounded derived category $\mc D^b_G (X)$, as in \cite{BeLu}, 
\cite[\S 1]{Lus-Cusp1} and \cite[\S 1]{Lus-Cusp2}. The formalism of \cite{BeLu} entails that 
(for non-discrete groups) this is not exactly the bounded derived category of the category of 
$G$-equivariant constructible sheaves on a $G$-variety $X$. Morphisms in 
$\mc D_G^b (X)$ are defined via a resolution of $X$ by $G$-varieties $Y$ as in \cite{BeLu}, and 
on each such $Y$ we use morphisms in a (non-equivariant) derived category of sheaves. In general,
checking that an object or a morphism belongs to $\mc D^b_G (X)$ can be reduced to two steps:
\begin{itemize}
\item show that it belongs to $\mc D^b_{G^\circ}(X)$,
\item show that the $G^\circ$-equivariant structure extends to a $G$-equivariant structure.
\end{itemize}
Typically the first step above is much harder than the second step, which is about abstract
group actions only.
The reason for this structure of $\mc D^b_G (X)$ is that from any $G^\circ$-resolution 
$Y$ of a $G$-variety $X$, one obtains a $G$-resolution $G \times^{G^\circ} Y$ of $X$, and those
resolutions suffice to study $\mc D_G^b (X)$ as constructed in \cite[\S 2]{BeLu}. This entails 
for instance that a morphism in $\mc D^b_G (X)$ is an isomorphism if and only it becomes an 
isomorphism in $\mc D^b_{G^\circ}(X)$.

Equivariant cohomology for objects of $\mc D^b_G (X)$ is defined via 
push-forward to a point, representing the result as a complex of sheaves on a 
classifying space $\mf B G$ for $G$ and then taking cohomology in $\mc D^b (\mf B G)$.
For more background we refer to \cite[Chapter 6]{Ach}.

We will use some notations and conventions from \cite{Lus-Cusp2}, in particular functors from or to 
$\mc D_G^b (X)$ are by default derived functors. Let $[n]$ be the functor that shifts degrees by $n$. 
For objects $A,B$ of $\mc D^b_G (X)$ and $n \in \Z$, we write 
\[
\Hom^n_{\mc D^b_G (X)} (A, B) = \Hom_{\mc D^b_G (X)} (A, B[n]) .
\]
In the case $A = B$ one obtains the graded algebra
\[
\mr{End}^*_{\mc D^b_G (X)}(A) = \bigoplus\nolimits_{n \in \Z} \Hom^n_{\mc D^b_G (X)} (A, A) .
\]

\subsection{Geometric construction of graded Hecke algebras} \
\label{par:geomConst}

Recall from \cite{AMS1} that a quasi-Levi subgroup of $G$ is a group of the form
$M = Z_G (Z(L)^\circ)$, where $L$ is a Levi subgroup of $G^\circ$. Thus
$Z (M)^\circ = Z(L)^\circ$ and $M \longleftrightarrow L = M^\circ$ is a bijection
between the quasi-Levi subgroups of $G$ and the Levi subgroups of $G^\circ$.

\begin{defn}\label{def:2.5}
A cuspidal quasi-support for $G$ is a triple $(M,\cC_v^M,q\cE)$ where:
\begin{itemize}
\item $M$ is a quasi-Levi subgroup of $G$;
\item $\cC_v^M$ is the $\Ad(M)$-orbit of a nilpotent element $v \in \mf m = \mathrm{Lie}(M)$.
\item $q \cE$ is a $M$-equivariant cuspidal local system on $\cC_v^M$, i.e. as an
$M^\circ$-equivariant local system it is a direct sum of cuspidal local systems.
\end{itemize}
\end{defn}
We denote the $G$-conjugacy class of $(M,\cC_v^M,q \cE)$ by $[M,\cC_v^M,q\cE]_G$. 
With this cuspidal quasi-support we associate the groups
\begin{equation}\label{eq:3.15}
N_G (q \cE) = \mathrm{Stab}_{N_G (M)}(q \cE) \quad \text{and} \quad
W_{q \cE} = N_G (q \cE) / M .
\end{equation}
Let $\mf g_N$ be the variety of nilpotent elements in the Lie algebra $\mf g = \mr{Lie}(G)$.
Cuspidal quasi-supports are useful to partition the set of $G$-equivariant local 
systems on Ad$(G)$-orbits in $\mf g_N$. Let $\cE$ be an irreducible constituent 
of $q\cE$ as $M^\circ$-equivariant local system on $\cC_v^M$ (which by the cuspidality of
$\cE$ equals the Ad$(M^\circ)$-orbit of $v$). Then
\[
W_\cE^\circ := N_{G^\circ}(M^\circ) / M^\circ \cong N_{G^\circ} (M^\circ) M / M
\]
is a subgroup of $W_{q\cE}$. It is normal because $G^\circ$ is normal in $G$.
Write $T = Z(M)^\circ$ and $\mf t = \mr{Lie}(T)$. It is known from \cite[Proposition 2.2]{Lus-Cusp1}
that $R(G^\circ,T) \subset \mf t^\vee$ is a root system with Weyl goup $W_\cE^\circ$.

Let $P^\circ$ be a parabolic subgroup of $G^\circ$ with Levi decomposition
$P^\circ = M^\circ \ltimes U$. The definition of $M$ entails that it normalizes $U$, so
\[
P := M \ltimes U
\]
is a again a group, a ``quasi-parabolic" subgroup of $G$. Since $W_\cE^\circ = W(R(G^\circ,T))$, 
all possible $P$ are conjugate by elements of $N_{G^\circ}(M^\circ)$. We put
\begin{align*}
& N_G (P,q\cE) = N_G (P,M) \cap N_G (q\cE) , \\
& \Gamma_{q\cE} = N_G (P,q\cE) / M .
\end{align*}
The same proof as for \cite[Lemma 2.1.b]{AMS2} shows that 
\begin{equation}\label{eq:3.1}
W_{q\cE} = W_\cE^\circ \rtimes \Gamma_{q\cE} .
\end{equation}
The $W_{q\cE}$-action on $T$ gives rise to an action of $W_{q\mc E}$ on $\mc O (\mf t) = S(\mf t^\vee)$.

We specify our parameters $k (\alpha)$. For $\alpha$ in the root system 
$R(G^\circ,T)$, let $\mf g_\alpha \subset \mf g$ be the associated eigenspace for the $T$-action.
Let $\Delta_P$ be the set of roots in $R(G^\circ,T)$ which are simple with respect to $P$.
For $\alpha \in \Delta_P$ we define $k (\alpha) \in \Z_{\geq 2}$ by
\begin{equation}\label{eq:1.5}
\begin{array}{ccc}
\ad (v)^{k (\alpha) - 2} : & \mf g_{\alpha} \oplus \mf g_{2 \alpha} \to \mf g_{\alpha} 
\oplus \mf g_{2 \alpha} & \text{ is nonzero,} \\
\ad (v)^{k (\alpha) - 1} : & \mf g_{\alpha} \oplus \mf g_{2 \alpha} \to \mf g_{\alpha} 
\oplus \mf g_{2 \alpha} & \text{ is zero.}
\end{array}
\end{equation}
Then $(k (\alpha) )_{\alpha \in \Delta_P}$ extends to a $W_{q\cE}$-invariant function $k : 
R(G^\circ,T)_{\mathrm{red}} \to \C$, where the subscript ``red" indicates the set of reduced (or
indivisible) roots. Let $\natural : (W_{q\cE} / W_\cE^\circ)^2 \to \C^\times$ be a 2-cocycle (to be 
specified later). To these data we associate the twisted graded Hecke algebra 
$\mh H (\mf t, W_{q\cE},k,\mb{r},\natural)$, as in Proposition \ref{prop:1.1}. 

To make the connection of the above twisted graded Hecke algebra with the cuspidal local system 
$q\cE$ complete, we involve the geometry of $G$ and $\mf g$. Write
\[
\mf t_\reg = \{ x \in \mf t : Z_{\mf g}(x) = \mf l \} \quad \text{and} \quad
\mf g_{RS} = \Ad (G) (\cC_v^M \oplus \mf t_\reg \oplus \mf u) .
\]
Consider the varieties
\begin{align*}
& \dot{\mf g} = \{ (X,gP) \in \mf g \times G/P : 
\Ad (g^{-1}) X \in \cC_v^M \oplus \mf t \oplus \mf u \} , \\
& \dot{\mf g}^{\circ} = \{ (X,gP) \in \mf g \times G^{\circ}/P^\circ : 
\Ad (g^{-1}) X \in \cC_v^M \oplus \mf t \oplus \mf u \}, \\
& \dot{\mf g}_{RS} = \dot{\mf g} \cap (\mf g_{RS} \times G/P) ,\\
& \dot{\mf g}_N = \dot{\mf g} \cap (\mf g_N \times G/P) .
\end{align*}
We let $G \times \C^\times$ act on these varieties by
\[
(g_1,\lambda) \cdot (X,gP) = (\lambda^{-2} \Ad (g_1) X, g_1 g P) .
\] 
By \cite[Proposition 4.2]{Lus-Cusp1} there is a natural isomorphism
\begin{equation}\label{eq:1.1}
H^*_{G \times \C^\times} (\dot{\mf g}) \cong \mc O (\mf t) \otimes_\C \C [\mb r] .
\end{equation}
The same calculation (omitting $\mf t$ from the definition of $\dot{\mf g}$) shows that
\begin{equation}\label{eq:1.50}
H^*_{G \times \C^\times} (\dot{\mf g}_N) \cong \mc O (\mf t) \otimes_\C \C [\mb r] .
\end{equation}
Consider the maps
\begin{equation}\label{eq:1.2}
\begin{aligned}
& \cC_v^M \xleftarrow{f_1} \{ (X,g) \in \mf g \times G : \Ad (g^{-1}) X \in 
\cC_v^M \oplus \mf t \oplus \mf u\} \xrightarrow{f_2} \dot{\mf g} , \\
& f_1 (X,g) = \mathrm{pr}_{\cC_v^M}(\Ad (g^{-1}) X) , \hspace{2cm} f_2 (X,g) = (X,gP) .
\end{aligned}
\end{equation}
The group $G \times \C^\times \times P$ acts on $\{ (X,g) \in \mf g \times G : \Ad (g^{-1}) X \in 
\cC_v^M \oplus \mf t \oplus \mf u\}$ by 
\[
(g_1,\lambda,p) \cdot (X,g) = ( \lambda^{-2} \Ad (g_1)X ,g_1 g p).
\]
Notice that the local system $q\cE$ on $\cC_v^M$ is $M \times \C^\times$-equivariant, because 
$\C^\times$ is connected and stabilizes 
nilpotent $M$-orbits. Further $f_1$ is constant on $G$-orbits, so $f_1^* q\cE$ is naturally a
$G \times \C^\times$-equivariant local system. Let $\dot{q\cE}$ be the unique $G \times 
\C^\times $-equivariant local system on $\dot{\mf g}$ such that $f_2^* \dot{q\cE} = f_1^* q\cE$. 
Let $\mr{pr}_1 : \dot{\mf g} \to \mf g$ be the projection on the first coordinate. 
When $G$ is connected, Lusztig \cite{Lus-Cusp2} has constructed graded Hecke algebras from
\[
K := \pr_{1,!} \dot{q\cE} \quad \in \mc D^b_{G \times \C^\times} (\mf g ) .
\]
For our purposes the pullback $K_N$ of $K$ to the nilpotent variety $\mf g_N \subset \mf g$ will be
more suitable than $K$ itself.

We can relate $\dot{\mf g}$ and $K$ to their versions for $G^\circ$, as follows. Write
\begin{equation}\label{eq:1.13}
G = \bigsqcup_{\gamma \in N_G (P,M) / M} G^\circ \gamma M / M
\quad \text{and} \quad G/P = \bigsqcup_{\gamma \in N_G (P,M) / M} G^\circ \gamma P / P .
\end{equation}
Then we can decompose
\begin{multline}\label{eq:1.21}
\dot{\mf g} \; = \; \bigsqcup\nolimits_{\gamma \in N_G (P,M) / M} \dot{\mf g}^\circ_\gamma \; := \; \\  
\hspace{-28mm} \bigsqcup\nolimits_{\gamma \in N_G (P,M) / M} 
\big\{ (X, g \gamma P) \in \dot{\mf g} : g \in G^\circ \big\} \cong \\
\bigsqcup_{\gamma \in N_G (P,M) / M} \hspace{-4mm} \big\{ (X,g \gamma P^\circ \gamma^{-1}) : 
X \in \mf g, g \in G^\circ / \gamma P^\circ \gamma^{-1}, \Ad (g^{-1}) X \in 
\Ad (\gamma) (\cC_v^M + \mf t + \mf u) \big\} .
\end{multline}
Here each term $\dot{\mf g}^\circ_\gamma$ is a twisted version of $\dot{\mf g}^\circ$. Consequently
$K$ is a direct sum of $G^\circ \times \C^\times$-equivariant subobjects, each of which is a twist
of the $K$ for $(G^\circ M, \cC_v^M, q\cE)$ by an element of $N_G (P,M) / M$.

Let $\dot{q \cE}_{RS}$ be the pullback of $\dot{q\cE}$ to $\dot{\mf g}_{RS}$. 
Let $\IC_{G \times \C^\times} (\mf g \times G/P,\dot{q \cE}_{RS})$ be the equivariant intersection 
cohomology complex determined by $\dot{q \cE}_{RS}$. This is just the usual intersection cohomology
complex in $\mc D^b (\mf g \times G/P)$, but with its $G \times \C^\times$-equivariant structure.
It is supported on the closure of $\dot{\mf g}_{RS}$ in $\mf g \times G/P$, a domain on which 
$\mr{pr}_1$ becomes proper. The map
\[
\mathrm{pr}_{1,RS} : \dot{\mf g}_{RS} \to \mf g_{RS} 
\]
is a fibration with fibre $N_G (M) / M$, so $(\pr_{1,RS})_! \dot{q\cE}_{RS}$ is a local 
system on $\mf g_{RS}$. It is shown in \cite[\S 3.4.a]{Lus-Cusp1} and 
\cite[Proposition 7.12.c]{Lus-Cusp2} that 
\begin{equation}\label{eq:1.4}
K \cong \mr{pr}_{1,!} \, \IC_{G \times \C^\times} (\mf g\times G/P,\dot{q \cE}_{RS}) \cong 
\IC_{G \times \C^\times}(\mf g, (\pr_{1,RS})_! \dot{q\cE}_{RS}) .
\end{equation}
Although in these references $G$ is assumed to be connected and $G \times \C^\times$-equivariance
is not mentioned, the entire argument in \cite[\S 3.4]{Lus-Cusp2} can be placed in the
appropriate $G \times \C^\times$-equivariant derived categories. The right hand side of \eqref{eq:1.4} 
shows that $K$ is a direct sum of simple perverse sheaves with support $\overline{\mf g_{RS}}$. 
Further, \cite[Lemma 5.4]{AMS1} and \cite[Proposition 7.14]{Lus-Cusp2} say that
\begin{equation}\label{eq:1.10}
\C[W_{q \cE},\natural_{q \cE}] \cong
\End^0_{\mc D^b_{G \times \C^\times}(\mf g_{RS})} \big( (\mathrm{pr}_{1,RS})_! \dot{q \cE}_{RS} \big) 
\cong \End^0_{\mc D^b_{G \times \C^\times}(\mf g)} ( K ) ,
\end{equation}
where $\natural_{q\cE} : (W_{q\cE} / W_\cE^\circ )^2 \to \C^\times$ is a suitable 2-cocycle. 
As in \cite[(8)]{AMS2}, we record the subalgebra of endomorphisms that stabilize Lie$(P)$: 
\begin{equation}\label{eq:3.60}
\End^0_{\mc D^b_{G \times \C^\times}(\mf g)} \big( \mathrm{pr}_{1,!} \dot{q \cE} \big)_P 
\cong \C[\Gamma_{q \cE},\natural_{q \cE}] .
\end{equation}
Now we associate to $(M,\cC_v^M,q\cE)$ the twisted graded Hecke algebra
\[
\mh H (G,M,q\cE) := \mh H (\mf t, W_{q\cE},k ,\mb r,\natural_{q \cE}) ,
\]
where the parameters $k (\alpha)$ come from \eqref{eq:1.5}. As in \cite[Lemma 2.8]{AMS2}, we
can regard it as 
\[
\mh H (G,M,q\cE) = \mh H (\mf t, W_\cE^\circ,k ,\mb r) \rtimes 
\End^0_{\mc D^b_{G \times \C^\times}(\mf g)} \big( \mathrm{pr}_{1,!} \dot{q \cE} \big)_P ,
\]
and then it depends canonically on $(G,M,q\cE)$. We note that \eqref{eq:3.1} implies
\begin{equation}\label{eq:3.13}
\mh H (G^\circ N_G (P,q\cE),M,q\cE) = \mh H (G,M,q\cE) .
\end{equation}
There is also a purely geometric realization of this algebra. For $\Ad (G) \times \C^\times$-stable 
subvarieties $\cV$ of $\mf g$, we define, as in \cite[\S 3]{Lus-Cusp1},
\begin{equation}\label{eq:1.27}
\begin{aligned}
& \dot{\cV} = \{ (X,gP) \in \dot{\mf g} : X \in \cV \} , \\
& \ddot{\cV} = \{ (X,gP,g' P) : (X,g P) \in \dot{\cV}, (X,g' P) \in \dot{\cV} \} .
\end{aligned}
\end{equation}
Let $q\cE^\vee$ be the dual equivariant local system on $\cC^M_v$, which is also cuspidal. 
It gives rise to $K^\vee = \pr_{1,!} \dot{q\cE}^\vee$, another equivariant intersection 
cohomology complex on $\mf g$. 
The two projections $\pi_{12}, \pi_{13} : \ddot{\cV} \to \dot{\cV}$ give rise 
to a $G \times \C^\times$-equivariant local system 
\[
\ddot{q\cE} = (\pi_{12} \times \pi_{13})^* \big( \dot{q\cE} \boxtimes \dot{q\cE}^\vee \big) 
\quad \text{on} \quad \ddot{\cV},
\]
which carries a natural action of \eqref{eq:1.1}. 
As in \cite{Lus-Cusp1}, the action of $\C [W_{q\cE},\natural_{q\cE}^{-1}]$ on $K^\vee$ leads to 
\begin{equation}\label{eq:1.3}
\text{actions of } \C [W_{q\cE},\natural_{q\cE}] \otimes \C [W_{q\cE},\natural_{q\cE}^{-1}] \text{ on } 
\ddot{q\cE} \text{ and on } H_j^{G^\circ \times \C^\times} \big( \ddot{\cV},\ddot{q\cE} \big) .
\end{equation}
In \cite{Lus-Cusp1} and \cite[\S 2]{AMS2} a left action $\Delta$ and a right action $\Delta'$ of 
$\mh H (G,M,q\cE)$ on $H_*^{G \times \C^\times}(\ddot{\mf g_N}, \ddot{q\cE})$ are constructed. 

\begin{thm}\label{thm:1.2}
\enuma{
\item The actions $\Delta$ and $\Delta'$ identify $H_*^{G \times \C^\times}(\ddot{\mf g}, 
\ddot{q\cE})$ and \\ $H_*^{G \times \C^\times}(\ddot{\mf g_N}, \ddot{q\cE})$ with the biregular 
representation of $\mh H (G,M,q\cE)$. 
\item Methods from equivariant cohomology provide natural isomorphisms of graded vector spaces
\[
\begin{array}{lll}
\End^*_{\mc D^b_{G \times \C^\times}(\mf g)}(K) & \cong & 
H_*^{G \times \C^\times}(\ddot{\mf g}, \ddot{q\cE}) ,\\
\End^*_{\mc D^b_{G \times \C^\times}(\mf g_N)}(K_N) & \cong & 
H_*^{G \times \C^\times}(\ddot{\mf g_N}, \ddot{q\cE}) .
\end{array}
\]
\item Parts (a) and (b) induce canonical isomorphisms of graded algebras
\[
\mh H (G,M,q\cE) \to \End^*_{\mc D^b_{G \times \C^\times}(\mf g)}(K) \to
\End^*_{\mc D^b_{G \times \C^\times}(\mf g_N)}(K_N) .
\]
}
\end{thm}
\begin{proof}
(a) When $G$ is connected, this is shown for $\ddot{\mf g_N}$ in \cite[Corollary 6.4]{Lus-Cusp1} 
and for $\ddot{\mf g}$ in the proof of \cite[Theorem 8.11]{Lus-Cusp2}, based on \cite{Lus-Cusp1}.
In \cite[Corollary 2.9 and \S 4]{AMS2} both are generalized to possibly disconnected $G$. \\
(b) For $(\mf g , K)$ with $G$ connected this is the beginning of the proof of 
\cite[Theorem 8.11]{Lus-Cusp2}. The same argument applies when $G$ is disconnected, and with
$(\mf g_N, K_N)$ instead of $(\mf g,K)$.\\
(c) In \cite[Theorem 8.11]{Lus-Cusp2} the first isomorphism is shown when $G$ is connected. 
Using parts (a,b) the same argument applies when $G$ is disconnected. Similarly we obtain
\[
\mh H (G,M,q\cE) \cong \End^*_{\mc D^b_{G \times \C^\times}(\mf g_N)}(K_N) .
\] 
These two graded algebra isomorphisms are linked via parts (a,b) and functoriality for the 
inclusion $\mf g_N \to \mf g$.
\end{proof}

\subsection{Semisimplicity of some complexes of sheaves} \
\label{par:semisimplicity}

For an alternative construction of $\dot{q\cE}$ and $K$, we consider the isomorphism of 
$G \times \C^\times$-varieties
\begin{equation}\label{eq:1.38}
\begin{array}{ccc}
G \times^P (\cC_v^M \oplus \mf t \oplus \mf u) & \to & \dot{\mf g} \\
(g,X) & \mapsto & (\mr{Ad}(g) X, gP) 
\end{array}.
\end{equation}
We note that the middle term in \eqref{eq:1.2} is isomorphic to 
$G \times ( \cC_v^M \oplus \mf t \oplus u )$ via the map 
$(X,g) \mapsto (g,\mr{Ad}(g^{-1})X)$. In these terms, \eqref{eq:1.2} becomes
\begin{equation}\label{eq:1.51}
\cC_v^M \xleftarrow{f'_1} G \times ( \cC_v^M \oplus \mf t \oplus \mf u ) \xrightarrow{f'_2} 
G \times^P (\cC_v^M \oplus \mf t \oplus \mf u),
\end{equation}
with the natural maps $f'_1$ and $f'_2$. We get $\dot{q\cE}$ as $G \times \C^\times$-equivariant 
local system on $G \times^P (\cC_v^M \oplus \mf t \oplus \mf u)$, satisfying $f^{'*}_2 \dot{q\cE} 
= f^{'*}_1 q\cE$. In this setup $\mr{pr}_1$ is replaced by 
\begin{equation}\label{eq:1.40}
\begin{array}{cccc}
\mu : & G \times^P (\cC_v^M \oplus \mf t \oplus \mf u) & \to & \mf g \\
 & (g,X) & \mapsto & \mr{Ad}(g)X 
\end{array}
\end{equation}
and then
\begin{equation}\label{eq:1.46}
K = \mu_! \, \dot{q\cE} .
\end{equation}
Recall that we defined $K_N$ as the pullback $K_N$ of $K$ to the variety $\mf g_N$, and that 
$K$ is a semisimple complex (that is, isomorphic to a direct sum of simple perverse sheaves,
maybe with degree shifts). We will prove that $K_N$ is also semisimple complex of sheaves. 
We write
\[
\dot{\mf g}_N = \dot{\mf g} \cap (\mf g_N \times G/P) .
\]
The maps \eqref{eq:1.2} restrict to
\begin{equation}\label{eq:1.36}
\cC_v^M \xleftarrow{f_{1,N}} \{ (X,g) \in \mf g_N \times G : \Ad (g^{-1}) X \in 
\cC_v^M \oplus \mf u\} \xrightarrow{f_{2,N}} \dot{\mf g}_N ,
\end{equation}
which allows us to define a local system $\dot{q\cE}_N$ on $\dot{\mf g}_N$ by 
$f_{2,N}^* \dot{q\cE} = f_{1,N}^* q\cE_N$. Then $\dot{q\cE}_N$ is the pullback of $\dot{q\cE}$ to
$\mf g_N$, because $f_{1,N}^* q\cE_N$ is the pullback of $f_1^* q\cE$. Let $\mr{pr}_{1,N}$ be
the restriction of $\mr{pr}_1$ to $\dot{\mf g}_N$. From the Cartesian diagram
\begin{equation}\label{eq:1.48}
\xymatrix{
\dot{\mf g}_N \ar[r] \ar[d]^{\mr{pr}_{1,N}} & \dot{\mf g} \ar[d]^{\mr{pr}_1} \\
\mf g_N \ar[r] & \mf g
}
\end{equation}
we see with base change \cite[Theorem 3.4.3]{BeLu} that 
\begin{equation}\label{eq:1.37}
\mr{pr}_{1,N,!} \, \dot{q\cE}_N \text{ equals the pullback } K_N \text{ of } K \text{ to } \mf g_N .
\end{equation}

\begin{prop}\label{prop:1.14}
There is a natural isomorphism
\[
K_N \cong \mr{pr}_{1,N,!} \, \IC_{G \times \C^\times} (\mf g_N \times G/P, \dot{q\cE}_N) .
\]
\end{prop}
\begin{proof}
Notice that the middle term in \eqref{eq:1.36} is isomorphic with $G \times \cC_v^M \oplus \mf u$
and that \eqref{eq:1.38} provides an isomorphism
\[
\dot{\mf g}_N \cong G \times^P (\cC_v^M \oplus \mf u ).
\]
With the commutative diagram
\begin{equation}\label{eq:1.39}
\xymatrix{
\cC_v^M \ar[d] & & \cC_v^M \oplus \mf u \ar[ll]^{\mr{pr}_{\cC_v^M}} \ar[d] \\
G \times^P \cC_v^M & & G \times^P (\cC_v^M \oplus \mf u) \ar[ll]^{\mr{id}_G \times \mr{pr}_{\cC_v^M}} 
}
\end{equation}
we can construct $\dot{q\cE}_N \in \mc D^b_{G \times \C^\times}(G \times^P (\cC_v^M \oplus \mf u))$
in two equivalent ways:
\begin{itemize}
\item pullback of $q\cE$ to $\cC_v^m \oplus \mf u$ (as $P \times \C^\times$-equivariant local system)
and then equivariant induction $\mr{ind}_{P \times \C^\times}^{G \times \C^\times}$ as in
\cite[\S 2.6.3]{BeLu};
\item equivariant induction $\mr{ind}_{P \times \C^\times}^{G \times \C^\times}$ of $q\cE$ to
$G \times^P \cC_v^M$ and then pullback to \\ $G \times^P (\cC_v^M \oplus \mf u)$.
\end{itemize}
In these terms 
\begin{equation}\label{eq:1.47}
K_N = \mu_{N,!} \, \dot{q\cE}_N ,
\end{equation}
where $\mu_N : G \times^P (\cC_v^M \oplus \mf u) \to \mf g_N$ is the restriction of \eqref{eq:1.40}.
Let $j_{\mf m_N} : \cC_v^M \to \mf m_N$ be the inclusion. Then
\begin{equation}\label{eq:1.41}
K_N = \mu_{N,!} (\mr{id}_G \times j_{\mf m_N} \times \mr{id}_{\mf u})_! \dot{q\cE}_N ,
\end{equation}
where now the domain of $\mu_N$ is $\mf m_N \oplus \mf u$.

Regarded as $M^\circ \times \C^\times$-equivariant local system on $\cC_v^M$, $q\cE$ is a direct sum
of irreducible cuspidal local systems $\cE$. Each of those $\cE$ is clean 
\cite[Theorem 23.1]{Lus-CharV}, which means that the natural maps 
\[
j_{\mf m_N,!} \cE \to \IC (\mf m_N, \cE) \to j_{\mf m_N,*} \cE 
\]
are isomorphisms in $\mc D^b_{M^\circ \times \C^\times}(\mf m_N)$.
Taking direct sums over the appropriate $\cE$, we find that the maps
\begin{equation}\label{eq:1.49}
j_{\mf m_N,!} q\cE \to \IC (\mf m_N, q\cE) \to j_{\mf m_N,*} q\cE .
\end{equation}
are isomorphisms in $\mc D^b_{M^\circ \times \C^\times}(\mf m_N)$ as well. In fact, 
since the maps in \eqref{eq:1.49} are $M \times \C^\times$-equivariant, they are also
isomorphisms in $\mc D^b_{M \times \C^\times}(\mf m_N)$ (see the remarks at the start
of Section \ref{sec:cuspidal}). In other words, $q\cE$ is also clean.

In the diagram \eqref{eq:1.39} the map $\mr{pr}_{\cC_v^M}$ extends naturally to
\[
\mr{pr}_{\mf m_N} : \mf m_N \oplus \mf u \to \mf m_N ,
\]
and both are trivial vector bundles. Hence (up to degree shifts)
\begin{multline}\label{eq:1.42}
\mr{pr}_{\mf m_N}^* j_{\mf m_N,!} q\cE = \mr{pr}_{\mf m_N}^* \IC_{P \times \C^\times} 
(\mf m_N, q\cE) = \mr{pr}_{\mf m_N}^* j_{\mf m_N,*} q\cE = \\
(j_{\mf m_N} \times \mr{id}_{\mf u})_* \mr{pr}_{\cC_v^M}^* q\cE =
\IC_{P \times \C^\times} (\mf m_N \oplus \mf u, \mr{pr}_{\cC_v^M}^* q\cE ) =
 (j_{\mf m_N} \times \mr{id}_{\mf u})_! \mr{pr}_{\cC_v^M}^* q\cE .
\end{multline}
The vertical maps in \eqref{eq:1.39} induce equivalences of categories 
$\mr{ind}_{P \times \C^\times}^{G \times \C^\times}$, which commute with the relevant functors 
induced by the horizontal maps in \eqref{eq:1.39}, so
\begin{equation}\label{eq:1.43}
\begin{aligned}
(\mr{id}_G \times j_{\mf m_N} \times \mr{id}_{\mf u})_! \dot{q\cE}_N & =
(\mr{id}_G \times j_{\mf m_N} \times \mr{id}_{\mf u})_! 
\mr{ind}_{P \times \C^\times}^{G \times \C^\times} \mr{pr}_{\cC_v^M}^* q\cE \\
& = \mr{ind}_{P \times \C^\times}^{G \times \C^\times} (j_{\mf m_N} \times \mr{id}_{\mf u})_!
\mr{pr}_{\cC_v^M}^* q\cE \\
& = \mr{ind}_{P \times \C^\times}^{G \times \C^\times} 
\IC_{P \times \C^\times} (\mf m_N \oplus \mf u, \mr{pr}_{\cC_v^M}^* q\cE ) \\
& = \IC_{G \times \C^\times} \big( G \times^P (\mf m_N \oplus \mf u), 
\mr{ind}_{P \times \C^\times}^{G \times \C^\times} \mr{pr}_{\cC_v^M}^* q\cE \big) \\
& = \IC_{G \times \C^\times} \big( G \times^P (\mf m_N \oplus \mf u), \dot{q\cE}_N \big) .
\end{aligned}
\end{equation}
Since $G \times^P (\mf m_N \oplus \mf u)$ is closed in 
$G \times^P \mf g_N$, the last expression is isomorphic with 
\begin{equation}\label{eq:1.55}
\IC_{G \times \C^\times} (G \times^P \mf g_N , \dot{q\cE}_N ).
\end{equation}
Via the isomorphism 
\begin{equation}\label{eq:1.57}
G \times^P \mf g_N \cong \mf g_N \times G/P 
\end{equation}
obtained from \eqref{eq:1.38} by restriction, \eqref{eq:1.55} becomes 
$\IC_{G \times \C^\times} (\mf g_N \times G/P, \dot{q\cE}_N )$. 
Combine that with \eqref{eq:1.41} and \eqref{eq:1.43}.
\end{proof}

The following method to prove semisimplicity of $K_N$ is based on the decomposition theorem for
perverse sheaves of geometric origin \cite[Th\'eor\`eme 6.2.5]{BBD}. The same method can be 
applied to $K$, using the first isomorphism in \eqref{eq:1.4}.

\begin{lem}\label{lem:1.17}
$K_N$ is a semisimple object of $\mc D^b_{G \times \C^\times} (\mf g_N)$.
\end{lem}
\begin{proof}
By construction every $M^\circ$-equivariant (cuspidal) local system on a Ad$(M^\circ)$-orbit in
$\mf m_N$ is geometric. The automorphism group $\mr{Aut}(M^\circ_{\mr{der}})$ of the derived 
subgroup of $M^\circ$ is algebraic and defined over $\Z$. The action of $M$ on $\mf m_N$ factors
through $\mr{Aut}(M^\circ_{\mr{der}})$, and hence the cuspidal local system $q\cE$ on $\cC_v^M$
is of geometric origin. 

Like for $M$, the automorphism group of $G^\circ_\der$ is algebraic and defined over $\Z$,
and the adjoint actions of $G$ and $P$ on $\mf g$ factor via that group.
Therefore not only $\pr_{\cC_v^M}^* q\cE$ but also 
\[
\dot{q\cE}_N = \mr{ind}_{P \times \C^\times}^{G \times \C^\times} \pr_{\cC_v^M}^* q\cE
\in \mc D^b_{G \times \C^\times} (G \times^P \mf m_N \oplus \mf u)
\]
is of geometric origin. As the isomorphism \eqref{eq:1.57} only involves $G$ via the adjoint action,
it follows that $\IC_{G \times \C^\times} (\mf g_N \times G / P, \dot{q\cE}_N)$ is of geometric 
origin as well. Since 
\[
\pr_{1,N} : \mf g_N \times G/P \to \mf g_N
\] 
is proper, we can apply the decomposition theorem for equivariant perverse sheaves 
\cite[\S 5.3.1]{BeLu} to Proposition \ref{prop:1.14}. This is based on the non-equivariant version 
from \cite[\S 6]{BBD}, and therefore requires objects of geometric origin.
\end{proof}

For compatibility with other papers we record that, by \eqref{eq:1.41}, \eqref{eq:1.42} and 
\eqref{eq:1.43}:
\begin{equation}\label{eq:1.44} 
\begin{aligned}
K_N & \cong \mu_{N,!} \mr{ind}_{P \times \C^\times}^{G \times \C^\times} 
\IC_{P \times \C^\times} (\mf m_N \oplus \mf u, \mr{pr}_{\cC_v^M}^* q\cE ) \\
& \cong \mu_{N,!} \mr{ind}_{P \times \C^\times}^{G \times \C^\times} \mr{pr}_{\mf m_N}^*
\IC_{P \times \C^\times} (\mf m_N, q\cE) \\
& \cong \mu_{N,!} (\mr{id}_G \times \mr{pr}_{\mf m_N})^*
\mr{ind}_{P \times \C^\times}^{G \times \C^\times} \IC_{P \times \C^\times} (\mf m_N, q\cE) .
\end{aligned}
\end{equation}
Like in \cite[\S 8.4]{Ach}, the diagram 
\begin{equation}\label{eq:1.52}
\mf m_N \to G \times^P \mf m_N \xleftarrow{\mr{id}_G \times \mr{pr}_{\mf m_N}} 
G \times^P (\mf m_N \oplus \mf u) \xrightarrow{\mu_N} \mf g_N
\end{equation}
gives rise to a parabolic induction functor
\begin{equation}\label{eq:1.53}
\mc I_{P \times \C^\times, \mf m_N}^{G \times \C^\times} = \mu_{N,!} (\mr{id}_G \times 
\mr{pr}_{\mf m_N})^* \mr{ind}_{P \times \C^\times}^{G \times \C^\times} :
\mc D^b_{P \times \C^\times} (\mf m_N) \to \mc D^b_{G \times \C^\times} (\mf g_N) .
\end{equation}
Since $U \subset P$ is contractible and acts trivially on $\mf m_N$, inflation along the
quotient map $P \to M$ induces an equivalence of categories
\[
\mc D^b_{P \times \C^\times} (\mf m_N) \cong \mc D^b_{M \times \C^\times} (\mf m_N) .
\]
With these notions \eqref{eq:1.44} says precisely that
\begin{equation}\label{eq:1.45}
K_N \cong \mc I_{P \times \C^\times, \mf m_N}^{G \times \C^\times} \, 
\IC_{M \times \C^\times} (\mf m_N, q\cE) .
\end{equation}
For later use we also mention the parabolic restriction functor
\begin{equation}
\mc R_{P \times \C^\times, \mf m_N}^{G \times \C^\times} = \big( \mr{ind}_{P \times \C^\times
}^{G \times \C^\times} \big)^{-1} (\mr{id}_G \times \mr{pr}_{\mf m_N})_* \mu_N^! :
\mc D^b_{G \times \C^\times} (\mf g_N) \to \mc D^b_{P \times \C^\times} (\mf m_N) .
\end{equation}
The arguments in Proposition \ref{prop:1.14} and \eqref{eq:1.44} admit natural analogues for $K$.
Namely, with the diagram
\[
\mf m_N \to G \times^P \mf m_N \xleftarrow{\mr{id}_G \times \mr{pr}_{\mf m_N}} 
G \times^P (\mf m_N \oplus \mf t \oplus \mf u) \xrightarrow{\mu} \mf g
\]
instead of \eqref{eq:1.52}, we get a functor similar to \eqref{eq:1.53}. That yields an isomorphism
\begin{equation}\label{eq:1.54}
K \cong \mu_! (\mr{id}_G \times \mr{pr}_{\mf m_N})^* 
\mr{ind}_{P \times \C^\times}^{G \times \C^\times} \IC_{P \times \C^\times} (\mf m_N, q\cE) .
\end{equation}
This also follows from \cite[Proposition 7.12]{Lus-Cusp2}, at least when $G$ is connected.

\subsection{Variations for centralizer subgroups} \
\label{par:centralizer}

Let $\sigma \in \mf t$, so that $M = Z_G (T) \subset Z_G (\sigma)$. We would like to compare Theorem
\ref{thm:1.2} with its version for $(Z_G (\sigma), M, q\cE)$. First we analyse the variety 
\[
(G/P)^\sigma := \{ g P \in G / P : \sigma \in \mr{Lie} (g P g^{-1}) \} .
\]
This is also the fixed point set of $\exp (\C \sigma)$ in $G/P$.
Let $Z_G^\circ (\sigma)$ be the connected component of $Z_G (\sigma)$.

\begin{lem}\label{lem:1.5}
For any $g P \in (G/P)^\sigma$, the subgroup $g P^\circ g^{-1} \cap Z_G^\circ (\sigma)$ of
$Z_G^\circ (\sigma)$ is parabolic.
\end{lem}
\begin{proof}
Consider the parabolic subgroup $P' := g P^\circ g^{-1}$ of $G^\circ$. Its Lie algebra $\mf p'$ 
contains the semisimple element $\sigma$, so there exists a maximal torus $T'$ of $P'$ with 
$\sigma \in \mf t'$. Let $M'$ be the unique Levi factor of $P'$ containing $T'$. The unipotent
radical $U'$ of $P'$ and the opposite parabolic $M' \bar{U}'$ give rise to decompositions of 
$Z(\mf m')$-modules
\[
\mf g = \bar{\mf u}' \oplus \mf p', \quad \mf p' = Z(\mf m') \oplus \mf m'_{\der} \oplus \mf u' .
\]
Since $Z(\mf m') \subset \mf t' \subset Z_{\mf g}(\sigma)$, these decompositions are preserved
by intersecting with $Z_{\mf g} (\sigma)$:
\[
Z_{\mf g}(\sigma) = Z_{\bar{\mf u}'}(\sigma) \oplus Z_{\mf p'}(\sigma), \quad 
Z_{\mf p'}(\sigma) = Z(\mf m') \oplus Z_{\mf m'_{\der}}(\sigma) \oplus Z_{\mf u'}(\sigma) .
\] 
This shows that $Z_{\mf g}(\sigma) \cap \mf p'$ is a parabolic subalgebra of $Z_{\mf g}(\sigma)$.
Hence $Z^\circ_G (\sigma) \cap P'$ is a parabolic subgroup of $Z_G^\circ (\sigma)$.
\end{proof}

The subgroup $Z_G (\sigma) \subset G$ stabilizes $(G/P)^\sigma$, so the latter is a union of
$Z_G (\sigma)$-orbits. 

\begin{lem}\label{lem:1.6}
The connected components of $(G/P)^\sigma$ are precisely its $Z_G^\circ (\sigma)$-orbits.
\end{lem}
\begin{proof}
Clearly every $Z_G^\circ (\sigma)$-orbit is connected. From \eqref{eq:1.13} we get an
isomorphism of varieties
\begin{equation}\label{eq:1.14}
G / P = \bigsqcup\nolimits_{\gamma \in N_G (M)/M} \gamma G^\circ P / P 
\cong \bigsqcup\nolimits_\gamma \gamma G^\circ / P^\circ .
\end{equation}
Here $Z_G^\circ (\sigma)$ acts on $\gamma G^\circ / P^\circ$ by
\[
z \cdot \gamma g P^\circ = \gamma (\gamma^{-1} z \gamma) g P^\circ ,
\]
so via conjugation by $\gamma^{-1}$ and the natural action of $\gamma^{-1} Z_G^\circ (\sigma) \gamma
= Z_G^\circ (\Ad (\gamma^{-1}) \sigma)$ on $G^\circ / P^\circ$. Taking $\exp (\C \sigma)$-fixed points
in \eqref{eq:1.14} gives
\begin{align*}
(G/P)^\sigma & \cong \bigsqcup\nolimits_\gamma (\gamma G^\circ / P^\circ )^\sigma \\
& = \bigsqcup\nolimits_\gamma \{ \gamma g P^\circ : g \in G^\circ, \sigma \in 
\mr{Lie}(\gamma g P^\circ g^{-1} \gamma^{-1}) \} \\
& = \bigsqcup\nolimits_\gamma \gamma \{ g P^\circ : g \in G^\circ, 
\Ad (\gamma^{-1}) \sigma \in \mr{Lie} (g P^\circ g^{-1}) \} \\
& = \bigsqcup\nolimits_\gamma \gamma (G^\circ / P^\circ )^{\Ad (\gamma^{-1}) \sigma} .
\end{align*}
This reduces the lemma to the case $G^\circ / P^\circ$, so to the connected group $G^\circ$.
For that we refer to \cite[Proposition 8.8.7.ii]{ChGi}. That reference is written for Borel subgroups, 
but with Lemma \ref{lem:1.5} the proof also applies to other conjugacy classes of parabolic subgroups.
\end{proof}

It is also shown in \cite[Proposition 8.8.7.ii]{ChGi} that every $Z_G^\circ (\sigma)$-orbit in
$(G/P)^\sigma$ is a submanifold and an irreducible component. 

\begin{lem}\label{lem:1.7}
There are isomorphisms of $Z_G (\sigma)$-varieties
\[
\bigsqcup_{w \in N_{Z_G (\sigma)}(M) \backslash N_G (M)} Z_G (\sigma) / Z_{w P w^{-1}}(\sigma)
\cong \bigsqcup_{w \in N_{Z_G (\sigma)}(M) \backslash N_G (M)} Z_G (\sigma) \cdot w P 
= (G / P)^\sigma .
\]
\end{lem}
\begin{proof}
By Lemma \ref{lem:1.6} there exist finitely many $\gamma \in G$ such that 
\begin{equation}\label{eq:1.15}
(G/P)^\sigma = \sqcup_\gamma Z_G^\circ (\sigma) \cdot \gamma P .
\end{equation}
Then the same holds with $Z_G (\sigma)$ instead of $Z_G^\circ (\sigma)$, and fewer $\gamma$'s. 
The $Z_G (\sigma)$-stabilizer of $\gamma P$ is
\[
\{ z \in Z_G (\sigma) : z \gamma P \gamma^{-1} = \gamma P \gamma^{-1} \} = 
Z_G (\sigma) \cap \gamma P \gamma^{-1} = Z_{\gamma P \gamma^{-1}}(\sigma) .
\]
That proves the lemma, except for the precise index set. 

Fix a maximal torus $T'$ of $Z_G^\circ (\sigma)$ with $T \subset T'$. Every parabolic subgroup of 
$G^\circ$ or $Z_G^\circ (\sigma)$ is conjugate to one containing $T'$. The $G^\circ$-conjugates of 
$P^\circ$ that contain $T'$ are the $w P^\circ w^{-1}$ with $w \in N_{G^\circ}(T')$, or equivalently with
\[
w \in N_{G^\circ}(T') / N_{P^\circ}(T') = N_{G^\circ}(T') / N_{M^\circ}(T') \cong 
N_{G^\circ}(M^\circ) / M^\circ .
\]
For $w,w' \in N_{G^\circ}(M^\circ)$, $w P^\circ$ and $w' P^\circ$ are in the same 
$Z_G^\circ (\sigma)$-orbit if and only if $w' w^{-1} \in N_{Z_G^\circ (\sigma)}(M^\circ)$. 
We find that
\begin{equation}\label{eq:1.19}
(G^\circ / P^\circ )^\sigma = \bigsqcup\nolimits_{w \in N_{Z_G^\circ (\sigma)}(M^\circ) \backslash
N_{G^\circ}(M^\circ)} Z_G^\circ (\sigma) \cdot w P^\circ .
\end{equation}
We note that the group 
\[
N_{G^\circ}(M^\circ) / M^\circ = N_{G^\circ}(T) / Z_{G^\circ}(T) = N_{G^\circ}(M) / M^\circ \cong
N_{G^\circ}(M) M / M
\]
normalises $P$. When we replace $G^\circ / P^\circ$ by $G / P$ in \eqref{eq:1.19}, the options for
$w$ need to be enlarged to $N_G (M) / M$. Next we replace $Z_G^\circ (\sigma)$ by $Z_G (\sigma)$,
so that $w P$ and $w' P$ are in the same $Z_G (\sigma)$-orbit if and only if $w' w^{-1} \in 
N_{Z_G (\sigma)}(M) / M$. Notice that $wP \in (G/P)^\sigma$ because
\[
\sigma \in \mf m = \mr{Lie}(w M w^{-1}) \subset \mr{Lie}(w P w^{-1}) .
\]
We conclude that
\[
(G / P )^\sigma = \bigsqcup_{w \in N_{Z_G^\circ (\sigma)}(M) \backslash
N_{G}(M)} Z_G^\circ (\sigma) \cdot w P = \bigsqcup_{w \in N_{Z_G (\sigma)}(M) \backslash
N_{G}(M)} Z_G (\sigma) \cdot w P . \qedhere
\]
\end{proof}

The fixed point set of $\exp (\C \sigma)$ in $\dot{\mf g}$ is
\[
\dot{\mf g}^\sigma = \dot{\mf g} \cap (Z_{\mf g} (\sigma) \times (G/P)^\sigma) =
\{ (X,gP) \in Z_{\mf g}(\sigma) \times (G/P)^\sigma : \Ad (g^{-1}) X \in \cC_v^M + \mf t + \mf u \} .
\]
Clearly $\dot{\mf g}^\sigma$ is related to $\dot{Z_{\mf g}(\sigma)}$ and to 
$\dot{Z_{\mf g}(\sigma)}^\circ$. With \eqref{eq:1.19} and \eqref{eq:1.21} we can make that precise:
\begin{align}
\label{eq:1.24} & \dot{\mf g}^\sigma = \bigsqcup_{w \in N_{Z_G^\circ (\sigma)}(M) \backslash N_{G}(M)}
\dot{Z_{\mf g}(\sigma)}^\circ_w = \bigsqcup_{w \in N_{Z_G (\sigma)}(M) \backslash N_{G}(M)}
\dot{Z_{\mf g}(\sigma)}_w \\
\nonumber & \dot{Z_{\mf g}(\sigma)}_w = \{ (X,g Z_{w P w^{-1}}(\sigma)) \in Z_{\mf g}(\sigma) \times
Z_G (\sigma) / Z_{w P w^{-1}}(\sigma) : \\
\nonumber & \hspace{2cm} \Ad (g^{-1})X \in \Ad (w) (\cC_v^M + \mf t + \mf u) \} .
\end{align}
Let $j' : \dot{\mf g}^\sigma \to \dot{\mf g}$ be the inclusion and let $\pr_1^\sigma$ be the restriction
of $\pr_1$ to $\dot{\mf g}^\sigma$. We define
\[
K_\sigma = (\pr_1^\sigma )_! j^{'*} \dot{q\cE} \in \mc D^b_{Z_G (\sigma) \times \C^\times}(Z_{\mf g}(\sigma)) .
\]
From \eqref{eq:1.24} we infer that $K_\sigma$ is a direct sum of the parts $K_{\sigma,w}$ (resp.
$K^\circ_{\sigma,w}$) coming from $\dot{Z_{\mf g}(\sigma)}_w$ (resp. from $\dot{Z_{\mf g}(\sigma)}_w^\circ$),
and each such part is a version of the $K$ for $Z_G (\sigma)$ (resp. for $Z_G^\circ (\sigma)$), twisted
by $w \in N_G (M) / M$.

These objects admit versions restricted to subvarieties of nilpotent elements, which we indicate by a
subscript $N$. In particular
\[
K_{N,\sigma} = (\pr_{1,N}^\sigma )_! j_N^{'*} \dot{q\cE}_N 
\in \mc D^b_{Z_G (\sigma) \times \C^\times}(Z_{\mf g}(\sigma)_N) 
\]
can be decomposed as a direct sum of subobjects $K_{N,\sigma,w}$ or $K_{N,\sigma,w}^\circ$.

\begin{lem}\label{lem:1.15}
The objects $K_\sigma, K_{\sigma,w} \in \mc D^b_{Z_G (\sigma) \times \C^\times} (Z_{\mf g}(\sigma))$ and \\ $K_{N,\sigma}, K_{N,\sigma,w} \in \mc D^b_{Z_G (\sigma) \times \C^\times} (Z_{\mf g}(\sigma)_N)$
are semisimple.
\end{lem}
\begin{proof}
We note that, like in \eqref{eq:1.38}, there is an isomorphism of 
$Z_G (\sigma) \times \C^\times$-varieties
\[
\dot{Z_G (\sigma)}_w \cong Z_G (\sigma) \times^{Z_{w P w^{-1}(\sigma)}} 
\big( \mr{Ad}(w) (\cC_v^M \oplus \mf t \oplus \mf u) \cap Z_{\mf g} (\sigma) \big) .
\]
Here $Z_{w P w^{-1}}(\sigma)$ is a quasi-parabolic subgroup of $G$ with quasi-Levi factor $M$ and
\[
\mr{Ad}(w) (\cC_v^M \oplus \mf t \oplus \mf u) \cap Z_{\mf g} (\sigma) =
\cC_v^M \oplus \mf t \oplus \big( \mr{Ad}(w) \mf u \cap Z_{\mf g} (\sigma) \big) 
\]
with $\mr{Ad}(w) \mf u \cap Z_{\mf g} (\sigma)$ the Lie algebra of the unipotent radical of 
$Z_{w P w^{-1}}(\sigma)$. Comparing that with the construction of $K$ in 
\eqref{eq:1.40}--\eqref{eq:1.46}, we deduce that $K_{\sigma,w}$ is the $K$ for the group $Z_G (\sigma)$ 
and the cuspidal local system $\mr{Ad}(w)_* q\cE$. As $K$ is semisimple, see \eqref{eq:1.4}, 
so is the current $K_{\sigma,w}$.

The same reasoning, now using \eqref{eq:1.47}, shows that $K_{N,\sigma,w}$ is the $K_N$ for 
$Z_G (\sigma)$ and $\mr{Ad}(w)_* q\cE$. By Proposition \ref{prop:1.14}.b, $K_{N,\sigma,w}$ is 
semisimple.

The objects $K_\sigma$ and $K_{N,\sigma}$ are direct sums of objects $K_{\sigma,w}$ and 
$K_{N,\sigma,w}$, so these are also semisimple.
\end{proof}

The above decompositions of $K_\sigma$ and $K_{N,\sigma}$ are the key to analogues of parts of
Paragraph \ref{par:geomConst} for $Z_G (\sigma)$.

\begin{lem}\label{lem:1.8}
Let $w,w' \in N_G (M) / M$. The inclusion $Z_{\mf g}(\sigma)_N \to Z_{\mf g}(\sigma)$ induces an 
isomorphism of graded $H^*_{Z_G^\circ (\sigma) \times \C^\times}(\pt)$-modules
\[
\Hom^*_{\mc D^b_{Z_G^\circ (\sigma) \times \C^\times}(Z_{\mf g}(\sigma))}
(K^\circ_{\sigma,w},K^\circ_{\sigma,w'}) \longrightarrow \Hom^*_{\mc D^b_{Z_G^\circ (\sigma) \times 
\C^\times}(Z_{\mf g}(\sigma)_N)} (K^\circ_{N,\sigma,w},K^\circ_{N,\sigma,w'}) .
\]
\end{lem}
\begin{proof}
Decompose $\dot{q\cE} |_{\dot{Z_{\mf g}(\sigma)}^\circ_w}$ as direct sum of irreducible 
$Z_G^\circ (\sigma) \times \C^\times$-equivariant local systems. Each summand is of the form
$\Ad (w)_* \dot{\cE}$, for an irreducible summand $\cE$ of $q\cE$ as $M^\circ$-equivariant
local system. Similarly we decompose $\dot{q\cE} |_{\dot{Z_{\mf g}(\sigma)}^\circ_{w'}}$
as direct sum of terms $\Ad (w')_* \dot{\cE}$. Like in the proof of Lemma \ref{lem:1.15}:
\begin{equation}\label{eq:1.29}
K^\circ_{\sigma,w} = \bigoplus\nolimits_{\cE} (\pr_{1,Z_{\mf g}(\sigma)})_! \Ad (w)_* \dot{\cE} ,
\end{equation}
and similarly for $K^\circ_{N,\sigma,w}, K^\circ_{\sigma,w'}$ and $K^\circ_{N,\sigma,w'}$.
A computation like the start of the proof of \cite[Theorem 8.11]{Lus-Cusp2} (already used in
Theorem \ref{thm:1.2}.b) shows that
\begin{multline}\label{eq:1.25}
\Hom^*_{\mc D^b_{Z_G^\circ (\sigma) \times \C^\times}(Z_{\mf g}(\sigma))} \big( (\pr_{1,Z_{\mf g}
(\sigma)})_! \Ad (w)_* \dot{\cE}, (\pr_{1,Z_{\mf g}(\sigma)})_! \Ad (w')_* \dot{\cE'} \big) \\
\cong H_*^{Z_G^\circ (\sigma) \times \C^\times} \big( \ddot{Z_{\mf g}(\sigma)}^\circ ,
i_\sigma^* \big( \Ad (w)_* \dot{\cE} \boxtimes \Ad (w')_* \dot{\cE'}^\vee \big) \big) . 
\end{multline}
Here $\ddot{Z_{\mf g}(\sigma)}^\circ = \dot{Z_{\mf g}(\sigma)}^\circ \times^{Z_{\mf g}(\sigma)} 
\dot{Z_{\mf g}(\sigma)}^\circ$ and 
\[
i_\sigma : \ddot{Z_{\mf g}(\sigma)}^\circ \to 
\dot{Z_{\mf g}(\sigma)}^\circ \times \dot{Z_{\mf g}(\sigma)}^\circ
\]
denotes the inclusion. The same applies with subscripts $N$:
\begin{multline}\label{eq:1.26}
\Hom^*_{\mc D^b_{Z_G^\circ (\sigma) \times \C^\times}(Z_{\mf g}(\sigma)_N)} \big( (\pr_{1,Z_{\mf g}
(\sigma)_N})_! \Ad (w)_* \dot{\cE}_N, (\pr_{1,Z_{\mf g}(\sigma)_N})_! \Ad (w')_* \dot{\cE'}_N \big) \\
\cong H_*^{Z_G^\circ (\sigma) \times \C^\times} \big( \ddot{Z_{\mf g}(\sigma)}_N^\circ ,
i_{N,\sigma}^* \big( \Ad (w)_* \dot{\cE} \boxtimes \Ad (w')_* \dot{\cE'}^\vee \big) \big) . 
\end{multline}
When $w = w'$ and $\cE = \cE'$, \eqref{eq:1.25} and \eqref{eq:1.26} are computed in \cite[Proposition 
4.7]{Lus-Cusp1}. In fact \cite[Proposition 4.7]{Lus-Cusp1} also applies in our more general setting, 
with different $\Ad (w)_* \dot{\cE}$ and $\Ad (w') \dot{\cE'}$. Namely, to handle those we add the 
argument from the proof of \cite[Proposition 2.6]{AMS2}, especially \cite[(11)]{AMS2}. That works for 
both $Z_{\mf g}(\sigma)$ and for $Z_{\mf g}(\sigma)_N$, and entails that there are natural isomorphisms
of graded $H^*_{Z_G^\circ (\sigma) \times \C^\times}(\pt)$-modules
\begin{equation}\label{eq:1.28}
\begin{array}{ccc}
H^*_{Z_G^\circ (\sigma) \times \C^\times}(\dot{Z_{\mf g}(\sigma)}^\circ) \otimes_\C H_0 
\big( \ddot{Z_{\mf g}(\sigma)}^\circ , i_\sigma^* \big( \Ad (w)_* \dot{\cE} \boxtimes 
\Ad (w')_* \dot{\cE'}^\vee \big) \big)  & \cong & \text{\eqref{eq:1.25}}, \\
H^*_{Z_G^\circ (\sigma) \times \C^\times}(\dot{Z_{\mf g}(\sigma)}^\circ) \otimes_\C 
H_0 \big( \ddot{Z_{\mf g}(\sigma)}_N^\circ , i_{N,\sigma}^* \big( \Ad (w)_* \dot{\cE} \boxtimes 
\Ad (w')_* \dot{\cE'}^\vee \big) \big) & \cong & \text{\eqref{eq:1.26}} .
\end{array} 
\end{equation}
Moreover, the proof of \cite[Proposition 4.7]{Lus-Cusp1} shows that the two lines of \eqref{eq:1.28} 
are isomorphic via the inclusion $Z_{\mf g}(\sigma)_N \to Z_{\mf g}(\sigma)$.
\end{proof}

Finally, we can generalize the second isomorphism in Theorem \ref{thm:1.2}.c.

\begin{prop}\label{prop:1.9}
The inclusion $Z_{\mf g}(\sigma)_N \to Z_{\mf g}(\sigma)$ induces a graded algebra isomorphism
\[
\End^*_{\mc D^b_{Z_G (\sigma) \times \C^\times}(Z_{\mf g}(\sigma))}(K_\sigma) \longrightarrow
\End^*_{\mc D^b_{Z_G (\sigma) \times \C^\times}(Z_{\mf g}(\sigma)_N)}(K_{N,\sigma}) . 
\]
\end{prop}
\begin{proof}
Take the direct sum of the instances of Lemma \ref{lem:1.8}, over all \\
$w,w' \in N_{Z_G^\circ (\sigma)} (M) \backslash N_G (M)$. 
By \eqref{eq:1.24}, that yields a natural isomorphism
\[
\End^*_{\mc D^b_{Z_G^\circ (\sigma) \times \C^\times}(Z_{\mf g}(\sigma))}(K_\sigma) \longrightarrow
\End^*_{\mc D^b_{Z_G^\circ (\sigma) \times \C^\times}(Z_{\mf g}(\sigma)_N)}(K_{N,\sigma}) . 
\]
Now we take $\pi_0 (Z_G (\sigma))$-invariants on both sides, that replaces $\End^*_{\mc D^b_{Z_G^\circ 
(\sigma) \times \C^\times} (?)}$ by $\End^*_{\mc D^b_{Z_G (\sigma) \times \C^\times} (?)}$.
\end{proof}

\section{Description of $\mc D^b_{G \times GL_1}(\mf g_N)$ with Hecke algebras}
\label{sec:description}

We want to make a (right) module category of $\mh H = \mh H (G,M,q\cE)$ 
equivalent with a category of equivariant constructible sheaves. In view of Theorem \ref{thm:1.2}, 
we should compare with $\mc D^b_{G \times GL_1} (\mf g_N,K_N)$, the triangulated subcategory of 
$\mc D^b_{G \times \C^\times} (\mf g_N)$ generated by the simple summands of the semisimple object 
$K_N$. Since that involves complexes of 
sheaves, we have to look at differential graded $\mh H$-modules. Recall that $\mh H$ has no terms 
in odd degrees, so that its differential can only be zero. Hence a differential graded 
$\mh H$-module $M$ is just a graded $\mh H$-module $\bigoplus_{n \in \Z} M_n$ with a differential 
$d_M$ of degree 1. 

As $\mc D^b_{G \times GL_1} (\mf g_N)$ is a derived category, we are lead to 
$\mc D (\mh H - \Mod_{dg})$, the derived category of differential graded right $\mh H$-modules. 
Its bounded version is $\mc D^b (\mh H - \Mod_{\mr{fgdg}})$, where the subscript stands for 
``finitely generated differential graded". We note that $\mh H -\Mod_{dg}$ is much smaller
than $\mh H -\Mod$, for instance the only irreducible $\mh H$-modules it contains are those
on which $\mc O (\mf t \oplus \C)$ acts via evaluation at $(0,0)$. In fact the triangulated
category  $\mc D^b (\mh H - \Mod_{\mr{fgdg}})$ is already generated by a single object, 
namely $\mh H$ \cite[Corollary 11.1.5]{BeLu}.

From another angle, we aim to analyse the entire category $\mc D^b_{G \times GL_1}(\mf g_N)$ 
in terms of suitable module categories of Hecke algebras. As there may be several Levi subgroups 
and equivariant cuspidal local systems involved, we will need a finite collection of such Hecke 
algebras. The motivating and archetypical example is \cite[\S 12]{BeLu}. There it is shown that, 
for a connected complex reductive group $G'$ with a Borel subgroup $B'$ containing a maximal 
torus $T'$, there are equivalences of triangulated categories
\begin{equation}\label{eq:3.9}
\mc D^b_{G'} (G' / B') \cong \mc D^b_{B'} (\mr{pt}) \cong 
\mc D^b (\mc O (\mr{Lie}(T')) - \Mod_{\mr{fgdg}}) .
\end{equation}
The isomorphism $\mh H (G,M,q\cE) \cong \End^*_{\mc D^b_{G \times \C^\times} (\mf g_N)} (K_N)$
from Theorem \ref{thm:1.2} gives rise to an additive functor
\begin{equation}\label{eq:3.3}
\begin{array}{ccc}
\mc D^b_{G \times \C^\times} (\mf g_N, K_N) & \to & \mc D^b (\mh H - \Mod_{\mr{fgdg}}) \\
S & \mapsto & \Hom^*_{\mc D^b_{G \times \C^\times} (\mf g_N)} (K_N,S) 
\end{array} . \vspace{-2mm}
\end{equation}
However, it is not clear whether this functor is triangulated or fully faithful (on an
appropriate subcategory). One problem is that $\mc D^b_{G \times \C^\times} (\mf g_N)$ does not 
arise as the (bounded) derived category of an abelian category, another that 
$\Hom^*_{\mc D^b_{G \times \C^\times} (\mf g_N)}$ is defined rather indirectly.

\subsection{Transfer to a ground field of positive characteristic} \
\label{par:transfer}

We will overcome the above problems by constructing a more subtle functor instead of \eqref{eq:3.3},
which will lead to an equivalence of categories. To this end, we will first transfer the entire
setup from the ground field $\C$ to a ground field of positive characteristic. 
All our varieties, algebraic groups and (complexes of) sheaves may be considered over an
algebraically closed ground field instead of $\C$, see \cite[\S 4.3]{BeLu}. In particular we can take
an algebraic closure $k_s$ of a finite field $\F_q$ whose characteristic $p$ is good for $G$, like in 
\cite{Rid,RiRu}. As we do not require that $G$ is connected, we decree that ``good" also means that $p$ 
does not divide the order of $\pi_0 (G)$. For consistency, we replace the variety $\mh A^1 (\C) = \C$ 
(on which $\mb r$ is the standard coordinate) by the affine space $\mh A^1$.
In our setting, the topology of the coefficient field $\C$ of our sheaves does not play a role.
Since $\Ql \cong \C$ as fields, we may just as well look at sheaves of $\Ql$-vector spaces everywhere.

This setup has the advantage that one can pass to varieties over finite fields, and to mixed 
(equivariant) sheaves. To emphasize that we consider an object with ground field $k_s$ we will 
sometimes add a subscript $s$. This notation comes from \cite[\S 6]{BBD}, where $k_s$ arises as the 
residue field for some discrete valuation ring, which relates $k_s$ to special fibres. In the remainder 
of this section we will regard $G$ as an algebraic group, and for an action of $G$ or $G \times GL_1$ 
we tacitly assume that these groups are considered over the same field as the varieties on which they act. 

To that end, and to get semisimplicty of $K_N$ from Lemma \ref{lem:1.17}, we assume that $G$ 
can be defined over a finite extension of $\Z$. That is hardly a restriction, since by 
Chevalley's construction it holds for any connected reductive group over an algebraically closed field.
For $G^\circ$ we may use Chevalley's $\Z$-model. Then reduction modulo $p$ and extension of scalars to
$k_s$ gives the corresponding Chevalley group $G^\circ_s$ over $k_s$. Similarly, from $\mf g_N$ over 
$\Z$ we obtain by reduction modulo $p$ and extension to $k_s$ the nilpotent variety $\mf g_{N,s}$ 
in the Lie algebra of $G^\circ_s$.

\begin{lem}\label{lem:3.5}
The classification of $G \times GL_1$-equivariant cuspidal local systems (with $\Ql$-coefficients) 
on $\mf g_N$ is the same for the two base fields $\C$ and $k_s$.
\end{lem}
\begin{proof}
The classification of equivariant cuspidal local systems supported on the variety of unipotent 
elements in $G^\circ$ \cite{Lus-Int} shows this for $G^\circ$. By \cite[\S 2.1.f]{Lus-Cusp1}, 
such local systems are automatically $G^\circ \times GL_1$-equivariant. With the natural bijection 
between the unipotent orbits in $G^\circ$ and the nilpotent orbits in $\mf g$ (recall that $p$ is 
good for $G^\circ$), we obtain the same result for $G^\circ \times GL_1$-equivariant cuspidal local 
systems supported  on $\mf g_N$. To get from there to $G \times GL_1$-equivariant 
cuspidal local systems boils down to extending representations of a finite group 
\[
\pi_0 (Z_{G^\circ}(X)) \cong \pi_0 (Z_{G^\circ \times GL_1}(X)) \qquad X \in \mf g_N
\] 
to a larger finite group 
\[
\pi_0 (Z_G (X)) \cong \pi_0 (Z_{G \times GL_1}(X)), 
\]
see \cite[\S 3]{AMS1}. In view of the short exact sequence
\[
1 \to \pi_0 (Z_{G^\circ}(N)) \to \pi_0 (Z_G (N)) \to G / G^\circ = \pi_0 (G) \to 1
\]
from \cite[(21)]{AMS1} and because $p$ does not divide $|\pi_0 (G)|$, this works in the same way
over $\C$ and over $k_s$.
\end{proof}

Of course Lemma \ref{lem:3.5} also applies to a quasi-Levi subgroup $M$ of $G$. That provides, for each
$q\cE, \mf m_N, K_N $ as before, versions $q\cE_s, \mf m_{N,s}, K_{N,s}$ with base field $k_s$.\\

For the transfer of $\mc D^b_{G \times GL_1} (\mf g_N)$ from $\C$ to $k_s$ we follow the strategy 
that was used to derive the decomposition theorem for equivariant perverse sheaves \cite[\S 5.3]{BeLu} 
from its non-equivariant version \cite[Th\'eor\`eme 6.2.5]{BBD}, which in turn relied on an analogue 
for varieties over finite fields. To apply the techniques from \cite[\S 6.1]{BBD}, it seems necessary 
that $G$ can be defined over a finite extension of $\Z$. In that case, all the varieties below can
also be defined over a finite extension of $\Z$.

Fix a segment $I \subset \Z$ and assume that $G \times GL_1$ is embedded in $GL_k$. It was noted in 
\cite[\S 3.1]{BeLu} that the variety $M_{|I|}$ of $k$-frames in the affine space $\mh A^{|I|+k}$ 
is an acyclic $G \times GL_1$-space. Then $G \times GL_1$ acts freely on $Q := M_{|I|} \times \mf g_N$ 
and the projection $p : Q \to \mf g_N$ is an $|I|$-acyclic resolution of $G \times GL_1$-varieties. Let 
\[
\overline{Q} = Q / (G \times GL_1) = (M_{|I|} \times \mf g_N) / (G \times GL_1)
\]
be the quotient variety. By \cite[\S 2.3.2]{BeLu}, $\mc D^I_{G \times GL_1}(\mf g_N)$ is naturally
equivalent with $\mc D^I (\overline Q |p)$, the full subcategory of $\mc D^I (\overline Q )$ made from 
all the objects that come from $\mf g_N$ via $p$ and $q : Q \to \overline{Q}$. Notice that all 
these objects are of geometric origin.

Let $\overline{Q}_{et}$ be $\overline{Q}$ with the \'etale topology. According to \cite[\S 6.1.2.B'']{BBD} 
there is a fully faithful embedding (for sheaves with $\Ql$-coefficients)
\[
\mc D^b (\overline{Q}_{et}) \longrightarrow \mc D^b (\overline Q) ,
\]
whose essential image consists of the objects of geometric origin. This restricts to an equivalence
of categories
\[
\mc D^b (\overline{Q}_{et} | p ) \longrightarrow \mc D^b (\overline Q | p) .
\]
Therefore we may replace the analytic topology on $\overline{Q}$ by the \'etale topology, and we
tacitly do that from now on.

For a variety $X$ defined over some finite extension of $\Z$, we denote by $X_s$ the base change to
$k_s$. According to \cite[\S 6.1.10 and p.159]{BBD} there is an equivalence of categories
\begin{equation}\label{eq:3.8}
\mc D^b_{\mc T,L} (\overline{Q}) \longleftrightarrow \mc D^b_{\mc T,L} (\overline{Q}_s) .
\end{equation}
Here $\mc T$ denotes an algebraic stratification of $\overline{Q}$ and $L$ means that for every
stratum a finite collection of irreducible smooth sheaves with $\Z / \ell \Z$-coefficients has been 
chosen. The subscript $\mc T,L$ refers to $(\mc T,L)$-constructible sheaves, as in \cite[\S 6.1.8]{BBD}.

Let $p_s : Q_s \to \mf g_{N,s}$ be the base change of $p$ and let $q_s : Q_s \to \overline{Q}_s$ 
be the base change of $q$, the quotient map for the free $G \times GL_1$-action. 

\begin{lem}\label{lem:3.10}
From \eqref{eq:3.8} one can obtain an equivalence of categories 
\[
\mc D^b (\overline Q | p) \longleftrightarrow \mc D^b (\overline{Q}_s | p_s) .
\]
\end{lem}
\begin{proof}
The $G \times GL_1$-orbits on $\mf g_N$ give a stratification of $Q = M_{|I|} \times \mf g_N$, and 
dividing out by $G \times GL_1$ produces a stratification $\mc T_I$ of $\overline{Q}$.
As $L_I$ we take those sheaves (of the indicated kind) on the strata of $\mc T_I$ that come from 
$\mf g_N$ via $p$ and $q$. Since there are only finitely many $G \times GL_1$-orbits in $\mf g_N$, 
this is a finite collection. Every $G \times GL_1$-equivariant sheaf on $\mf g_N$ is 
constructible with respect to the orbit stratification, so $L_I$ provides enough objects to make 
$\mc D^b (\overline Q | p)$ and $\mc D^b (\overline{Q}_s | p_s)$ constructible with respect to
$(\mc T_I,L_I)$.

On each stratum in $\mf g_N$ or $\mf g_{N,s}$ the isomorphism classes 
of irreducible equivariant sheaves with $\Z / \ell \Z$-coefficients can be put in bijection with 
the isomorphism classes of irreducible representations of the equivariant fundamental group in 
$G \times GL_1$. Hence $L_I$ has the same property for $\overline{Q}$ and $\overline{Q}_s$, 
with respect to $p,q$ and $p_s,q_s$. 

However, maybe $\mc T_I$ does not have the right geometric properties to apply 
\cite[Lemma 6.1.9 and \S 6.1.10]{BBD}. As in \cite[p. 155]{BBD}, we may refine the situation, 
by passing to larger finite extension $A$ of $\Z$ and a finer stratification $\mc T$ defined over 
$A$, whose fibers are smooth and geometrically connected. For $L$ we may take a finite 
collection as on \cite[p. 156]{BBD}, such that $(\mc T_I,L_I)$-constructibility implies 
$(\mc T,L)$-constructibility. Then \cite[\S 6.1.10]{BBD} may be used.

For this $(\mc T,L)$, $\mc D^b_{\mc T,L} (\overline{Q})$ contains $\mc D^b (\overline{Q} | p)$
as the full subcategory of objects whose pullback along $q : Q \to \overline{Q}$ is isomorphic to 
the pullback along 
\[
p : M_{|I|} \times \mf g_N \to \mf g_N
\] 
of an object of $\mc D^b (\mf g_N)$. Under the equivalence \eqref{eq:3.8}, this corresponds to 
the full subcategory of $\mc D^b_{\mc T,L} (\overline{Q}_s)$ formed by the 
objects whose pullback along $q_s$ is isomorphic to the pullback along 
\[
p_s : M_{|I|,s} \times \mf g_{N,s} \to \mf g_{N,s}
\] 
of an object of $\mc D^b (\mf g_{N,s})$. In view of the properties of $L,L_I$ with respect to 
$p_s,q_s$, this full subcategory is exactly $\mc D^b (\overline{Q}_s | p_s)$. 
\end{proof}

Lemma \ref{lem:3.10} can be restricted from $\mc D^b (\overline Q |p)$ to $\mc D^I 
(\overline Q |p)$. From that and \cite[\S 2.3.2]{BeLu} we obtain equivalences of categories
\begin{equation}\label{eq:3.10}
\mc D^I_{G \times GL_1} (\mf g_N) \longleftrightarrow \mc D^I (\overline Q |p) 
\longleftrightarrow \mc D^I (\overline{Q}_s | p_s) \longleftrightarrow 
\mc D^I_{G \times GL_1} (\mf g_{N,s}) .
\end{equation}
Recall $K_{N,s}$ from Lemma \ref{lem:3.5} and the subsequent remark.

\begin{thm}\label{thm:3.11}
Assume that $G$ can be defined over a finite extension of $\Z$.
There are equivalences of categories
\[
\begin{array}{l@{\qquad}lll}
(a) & \mc D^b_{G \times GL_1}(\mf g_N) & \longleftrightarrow & \mc D^b_{G \times GL_1}(\mf g_{N,s}) ,\\
(b) & \mc D^b_{G \times GL_1}(\mf g_N, K_N) & \longleftrightarrow &
\mc D^b_{G \times GL_1}(\mf g_{N,s}, K_{N,s}) .
\end{array}
\]
\end{thm}
\begin{proof}
(a) The equivalences \eqref{eq:3.10} can be achieved for any segment $I \subset \Z$, but we note that
$Q$ and $\overline{Q}$ depend on $|I|$. For a larger segment $I' \supset I$, the $G \times GL_1$-space 
$Q' = M_{|I'|} \times \mf g_N$ contains $Q = M_{|I|} \times \mf g_N$. Hence \eqref{eq:3.10} for
$I$ embeds canonically in \eqref{eq:3.10} for $I'$. The enables us to take the limit of all these
instances of \eqref{eq:3.10}. \\
(b) It remains to show that part (a) sends $K_N$ to $K_{N,s}$. By the definition of equivariant derived 
categories \cite{BeLu}, that means that we have to compare $p^* K_N \in \mc D^b_{G \times GL_1} (Q)$ 
with $p_s^* K_{N,s} \in \mc D^b_{G \times GL_1} (Q_s)$. Let $K'_N$ be the image of $p^* K_N$ in 
$\mc D^b_{\mc T,L}(\overline{Q})$ via the proof of Lemma \ref{lem:3.10}, so $q^* K'_N \cong p^* K_N$. 
We define $K'_{N,s}$ analogously. Since 
\[
q^* : \mc D^b (\overline{Q}) \to \mc D^b_{G \times GL_1}(Q)
\]
is an equivalence of categories, it suffices to show that \eqref{eq:3.8} sends $K'_N$ to $K'_{N,s}$.

By \cite[\S 6.1.7 and 6.1.10]{BBD} equivalences of the kind \eqref{eq:3.8} respect the 
usual derived functors associated to maps between varieties, provided $(\mc T,L)$ may be enlarged
depending on the object under consideration.

We claim that the step from $\mc D^b (\overline Q)$ to $\mc D^b_{G \times GL_1} (\mf g_N)$ respects 
equivariant induction $\mr{ind}_{P \times GL_1}^{G \times GL_1}$. To shorten the notations we
check this for $\mr{ind}_P^G$. Let $Y$ be a $P$-variety with a resolution $X$ and let 
$\mc F \in \mc D^b_P (Y)$ come from $X / P$. From the diagram
\[
Y \longleftarrow X \longrightarrow X / P = ( G \times^P X) / G \longleftarrow G \times^P X 
\longrightarrow G \times^P Y
\]
we see that $Y$ and $\mr{ind}_P^G (\mc F) \in \mc D^b_G (G\times^P Y)$ come from the same object of 
$\mc D^b (X/P)$. It follows that equivalences of the kind \eqref{eq:3.8}, when retracted to 
equivariant derived categories like in Lemma \ref{lem:3.10}, also respect equivariant induction.

We resort to the expression $K_N = \mu_{N,!} \dot{q\cE_N}$ from \eqref{eq:1.47}. That and the
constructions described in terms of \eqref{eq:1.39} only use derived functors of the kind discussed 
above. In this way we reduce our task from $K_N$ to $q\cE \in \mc D^b_{M \times GL_1} (\cC_v^M)$. 

The construction of the equivalence \eqref{eq:3.8} in \cite[\S 6.1.8--6.1.9]{BBD} uses only base
change maps, it goes via $\mc D^b_{\mc T,L}(\overline{Q}_A)$ for a suitable subring $A$ of $\C$.
Knowing that, Lemma \ref{lem:3.10} and its proof entail that $q\cE$ is sent to $q\cE_s$ by the
version of \eqref{eq:3.10} for $M$. This proves that the equivalence of categories from part (a)
sends $K_N$ to $K_{N,s}$. 
\end{proof}

\subsection{Orthogonal decomposition} \
\label{par:orthogonal}

From \cite{RiRu1} it can be expected that $\mc D^b_{G \times GL_1}(\mf g_{N,s})$ decomposes as an 
orthogonal direct sum of full subcategories of the form $\mc D^b_{G \times GL_1}(\mf g_{N,s} ,K_{N,s})$. 
Here orthogonality means that there are no nonzero morphisms between objects from different summands.
In view of Theorem \ref{thm:3.11}, the same should work over the base field $\C$.

We start with an orthogonality statement on the cuspidal level. Let $\cC_v^M, \cC_{v'}^M$ be 
nilpotent Ad$(M)$-orbits in $\mf m$ and let $q\cE, q\cE$ be $M$-equivariant irreducible 
cuspidal local systems on respectively $\cC_v^M$ and $\cC_{v'}^M$. As noted in 
\cite[\S 2.1.f]{Lus-Cusp1}, $q\cE$ and $q\cE'$ are automatically $M \times GL_1$-equivariant. Let 
$\IC (\mf m_N, q\cE)$ and $\IC (\mf m_N,q\cE' )$ be the associated $M \times GL_1$-equivariant 
intersection cohomology complexes. Notice that $\IC_{M \times GL_1}(\mf m_N,q\cE)$ is the version
of $K_N$ for $M$. 

\begin{lem}\label{lem:3.4}
Suppose that $\cC_v^M \neq \cC_{v'}^M$ or that $\cC_v^M = \cC_{v'}^M$ and $q\cE, q\cE'$ are not
isomorphic in $\mc D^b_{M \times GL_1}(\cC_v^M)$. Then
\[
\Hom^*_{\mc D^b_{M \times GL_1}(\mf m_N)} \big( \IC (\mf m_N, q\cE), \IC (\mf m_N,q\cE' )\big) = 0 .
\]
\end{lem}
\begin{proof}
Suppose that the given Hom-space is nonzero. By the remarks at the start of Section 
\ref{sec:cuspidal}, this remains true if we replace $M \times GL_1$ by its neutral component 
$M^\circ \times GL_1$. As $M^\circ \times GL_1$-equivariant local systems, we can decompose
$q\cE = \bigoplus_i \cE_i$ and $q\cE' = \bigoplus_j \cE'_j$, where the $\cE_i$ and the 
$\cE'_j$ are irreducible and cuspidal. Then 
\begin{equation}\label{eq:3.16}
\bigoplus\nolimits_{i,j} \Hom^*_{\mc D^b_{M^\circ \times GL_1}(\mf m_N)} 
\big( \IC (\mf m_N, \cE_i), \IC (\mf m_N,\cE'_j )\big) \neq 0 .
\end{equation}
Pick indices $i,j$ for which the corresponding summand of \eqref{eq:3.16} is nonzero.
The results in \cite[Appendix A]{RiRu1} are formulated for a connected reductive group and
its centre, but they also work for a $M^\circ \times GL_1$-variety and with $Z(M^\circ)$ instead
of the centre of $M^\circ \times GL_1$. Thus, we can analyse the $Z(M^\circ)$-characters of the
$M^\circ \times GL_1$-equivariant perverse sheaves $\cE_i$ and $\cE'_j$ as in 
\cite[Appendix A]{RiRu1}. Notice that Theorem \ref{thm:3.11} enables us to apply 
this with base field $\C$ as well. From \cite[Proposition A.8]{RiRu} and \eqref{eq:3.16} one 
concludes that $\cE_i$ and $\cE'_j$ have the same $Z(M^\circ)$-character. According to 
\cite[p. 205]{Lus-Int}, that implies that $\cE_i$ is isomorphic to $\cE'_j$ in 
$\mc D^b_{M^\circ}(\mf m_N)$. 
Hence $\cC_v^{M^\circ} = \cC_{v'}^{M^\circ}$, so we may (and will) assume that $v = v'$. 

Recall from 
\eqref{eq:1.49} that $q\cE$ and $q\cE'$ are clean (on $\mf m_N$). With adjunction we compute
\begin{equation}\label{eq:3.11}
\begin{aligned}
& \Hom^*_{\mc D^b_{M \times GL_1}(\mf m_N)} \big( \IC (\mf m_N, q\cE), \IC (\mf m_N,q\cE' )\big) = \\
& \Hom^*_{\mc D^b_{M \times GL_1}(\mf m_N)} \big( \IC (\mf m_N, q\cE), j_{\mf m_N,*} q\cE' \big) = \\
& \Hom^*_{\mc D^b_{M \times GL_1}(\cC_{v}^M)} \big( j^*_{\mf m_N} \IC (\mf m_N, q\cE), q\cE' \big) =
 \Hom^*_{\mc D^b_{M \times GL_1}(\cC_{v}^M)} \big( q\cE, q\cE' \big) .
\end{aligned}
\end{equation}
Let $\rho, \rho' \in \Irr \big( \pi_0 (Z_{M \times GL_1}(v)) \big)$ be the images of $q\cE$ and $q\cE'$
under the equivalence of categories
\[
\mc D^b_{M \times GL_1}(\cC_{v}^M) \cong \mc D^b_{Z_{M \times GL_1}(v)} (\{v\}) .
\]  
Then \eqref{eq:3.11} reduces to
\[
\Hom^*_{\mc D^b_{Z_{M \times GL_1}(v)} (\{v\}) } \big( \rho, \rho' \big) =
\Hom_{\pi_0 (Z_{M \times GL_1} (v))} (\rho, \rho') .
\]
Since $q\cE$ and $q\cE'$ are not isomorphic, $\rho$ and $\rho'$ are not isomorphic, and this expression
vanishes. That contradicts the assumption at the start of the proof.
\end{proof}

Consider the collection of all cuspidal quasi-supports $(M,\cC_v^M,q\cE)$ for $G$. Since each
$\mf m_N$ admits only very few irreducible $M$-equivariant cuspidal local systems
\cite[Introduction]{Lus-Int}, there are only finitely many $G$-conjugacy classes of cuspidal
quasi-supports for $G$. We note that by Lemma \ref{lem:3.5} the classification is the same over
the ground fields $\C$ and $k_s$. Each such conjugacy class $[M,\cC_v^M,q\cE ]_G$ gives rise to a full
triangulated subcategory
\[
\mc D^b_{G \times GL_1}(\mf g_N ,K_N) = 
\mc D^b_{G \times GL_1} \big( \mf g_N ,\mc I_{P \times \C^\times, \mf m_N}^{G \times \C^\times} 
\IC_{M \times GL_1}(\mf m_N,q\cE) \big) ,
\]
see \eqref{eq:1.45} for the equality. 

\begin{thm}\label{thm:3.6}
There is an orthogonal decomposition
\[
\mc D^b_{G \times GL_1} (\mf g_N) = 
\bigoplus\nolimits_{[M,\cC_v^M,q\cE]_G} \mc D^b_{G \times GL_1} \big( \mf g_N ,\mc I_{P \times 
\C^\times, \mf m_N}^{G \times \C^\times}  \IC_{M \times GL_1}(\mf m_N,q\cE) \big) .
\]
The same holds over the ground field $k_s$.
\end{thm}
\begin{proof}
Over $k_s$ is the translation of \cite[Theorem 3.5]{RiRu1} to our setting. Almost the entire proof in 
\cite[\S 2--3]{RiRu1} is valid in our generality, only the argument with central characters (near the 
end of the proof of \cite[Theorem 3.5]{RiRu1}) does not work any more. We extend that to our setting 
with Lemma \ref{lem:3.4}. 

We may pass to the base field $\C$ with Theorem \ref{thm:3.11}. 
\end{proof}

\subsection{An equivalence of triangulated categories} \
\label{par:triangulated}

We aim to show that $\mc D^b_{G \times GL_1}(\mf g_{N,s}, K_{N,s})$ is equivalent with 
$\mc D^b (\mh H - \Mod_{\mr{fgdg}})$, as triangulated categories. We follow the strategy outlined in 
\cite[\S 4]{RiRu}, based on \cite{Rid}, but with $G \times GL_1$ instead of $G$. We need the following 
objects as substitutes for objects appearing in the derived generalized Springer correspondence from 
\cite{Rid,RiRu}:
\[
\begin{array}{l|l|l}
\text{our setting} & \text{setting from \cite{RiRu}} & \text{setting from \cite{Rid}} \\
\hline
\mf g_N & \mc N & \mc N \\
\IC (\mf m_N, q\cE) & \IC_{\mb c} & \IC (\{0\}, \overline{\Q_\ell}) \\
K_N & \mh A_{\mb c} & \mb A \\
\C [W_{q\cE} ,\natural_{q\cE}] & \Ql [W(L)] & \Ql [W] \\
H^*_{G \times \C^\times} (\dot{\mf g_N}) \cong \mc O (\mf t \oplus \C) & H_L (\mc O_L) \cong S \mf z^* &
H^*_G (G/B) \cong S \mf h^* \\
\mh H (G,M,q\cE) & \Ql [W(L)] \ltimes S \mf z^* &
\mc A_G = \Ql [W(L)] \ltimes S \mf h^* \\
\mc I_{P \times \C^\times, \mf m_N}^{G \times \C^\times}  & \mc I_P^G & \Psi \\
\mc R_{P \times \C^\times, \mf m_N}^{G \times \C^\times}  & \mc R_P^G & \Phi \\
\mc D^b_{M \times \C^\times} (\mf m_N, \IC (\mf m_N, q\cE)) & \mc D^b_L (\mc N_L ,\IC_{\mb c}) \cong
\mc D^b_Z (\pt) & \mc D^b_G (G/B) \cong \mc D_T^b (\pt) 
\end{array}
\]
For the third till sixth lines of the table we refer to, respectively, \eqref{eq:1.45},
\eqref{eq:1.10}, \eqref{eq:1.50} and Theorem \ref{thm:1.2}.c. To justify the last line of the table, 
we note that the proof of \cite[Lemma 2.3 and Proposition 2.4]{RiRu} shows that
\begin{equation}\label{eq:3.2}
\mc D^b_{M \times \C^\times} (\mf m_N, \IC (\mf m_N, q\cE)) \cong
\mc D^b_{Z(M)^\circ \times \C^\times} (\pt) = \mc D^b_{T \times \C^\times} (\pt) .
\end{equation}
As explained in \cite[\S 3.2]{RiRu}, the $M$-equivariant cuspidal local system $q\cE_s$ on 
$\mc C_{v,s}^M \subset \mf m_{N,s}$ admits a version over a finite field $\F_q$, such that a Frobenius
element of Gal$(k_s / \F_q)$ acts trivially (after base change to $k_s$). Then everything can be set 
up over $\F_q$ with mixed sheaves, as in \cite[\S 4--5]{Rid}. Like in \cite{Rid,RiRu}, we indicate the
analogous objects over $\F_q$ with a subscript $\circ$. 

\begin{lem}\label{lem:3.8}
There are algebra isomorphisms 
\[
\Ql [W_{q\cE}, \natural_{q\cE}] \cong \End_{\mc D^b_{G \times GL_1} (\mf g_{N,s})}(K_{N,s}) \cong
\End_{\mc D^b_{G \times GL_1} (\mf g_{N,\circ})}(K_{N,\circ}) ,
\]
and the Frobenius action on the middle term is trivial.
\end{lem}
\begin{proof}
The first isomorphism can be shown in the same way as \eqref{eq:1.10}. This uses among others a notion
of good and bad $P-P$ double cosets in $G$, or equivalently $G$-orbits in $G/P \times G/P$, extended
from \cite{Lus-Cusp1} to disconnected $G$ in \cite[proof of Proposition 2.6]{AMS2}. With this notion,
\cite[\S 3.3]{RiRu} generalizes to disconnected $G$ over $\F_q$. In particular 
\cite[Proposition 3.7]{RiRu} shows that 
\[
\dim_{\Ql} \End_{\mc D^b_{G \times GL_1} (\mf g_{N,\circ})}(K_{N,\circ}) = |W_{q\cE}| .
\]
Next \cite[Proposition 3.9]{RiRu} proves that Gal$(k_s/\F_q)$ acts trivially on\\
$\End_{\mc D^b_{G \times GL_1} (\mf g_{N,s})}(K_{N,s})$, and shows that the natural map 
\[
\End_{\mc D^b_{G \times GL_1} (\mf g_{N,\circ})}(K_{N,\circ}) \to
\End_{\mc D^b_{G \times GL_1} (\mf g_{N,s})}(K_{N,s}) 
\]
induced by base change is an algebra isomorphism.
\end{proof}

Like in Theorem \ref{thm:1.2}, we obtain
\[
\End^*_{\mc D^b_{G \times GL_1} (\mf g_{N,s})} (K_{N,s}) = \mh H_{\Ql} =
\mh H_{\Ql} (G,M,q\cE) ,
\]
the version of $\mh H$ with scalars $\Ql$ instead of $\C$. With that settled, the proof of \cite[Theorem 
4.1]{RiRu} applies to $(\mf g_{N,\circ}, K_{N,\circ})$. It provides a triangulated category 
\[
K^b \mr{Pure}_{G \times GL_1} (\mf g_{N,\circ}, K_{N,\circ}) ,
\]
which is a mixed version of $\mc D^b_{G \times GL_1} (\mf g_{N,s}, K_{N,s})$ in the sense of 
\cite[Definition 4.2]{Rid}. Next \cite[Theorem 4.2]{RiRu} and \cite[\S 6]{Rid} generalize readily
to our setting (but with objects over the ground field $\F_q$). 
In particular these entail an equivalence of triangulated categories
\begin{equation}\label{eq:3.4}
K^b \mr{Pure}_{G \times GL_1} (\mf g_{N,\circ}, K_{N,\circ}) \cong
\mc D^b (\mh H_{\Ql} -\Mod_{\mr{fgdg}}) .
\end{equation}
Recall the notion of Koszulity for differential graded algebras from \cite{BGS}.

\begin{lem}\label{lem:3.1}
\enuma{
\item The algebra $\mh H_{\Ql}$ is Koszul. 
\item The Koszul dual $E(\mh H_{\Ql})$ of $\mh H_{\Ql}$ is a finite dimensional graded algebra.
}
\end{lem}
\begin{proof}
(a) Consider the degree zero part $\mh H_{\Ql,0} = \Ql [W_{q\cE},\natural_{q\cE}]$ as 
$\mh H_{\Ql}$-module, annihilated by all terms of positive degree. We have to find a resolution of 
$\mh H_{\Ql,0}$ by projective graded modules $P^n$, such that each $P^n$ is generated by its part
in degree $n$. We will use that the multiplication map
\[
\mh H_{\Ql,0} \otimes_{\Ql} \mc O (\mf t \oplus \mh A^1) \to \mh H_{\Ql}
\]
is an isomorphism of graded vector spaces. Start with the standard Koszul resolution for 
$\mc O (\mf t \oplus \mh A^1)$:
\[
\Ql \leftarrow \mc O (\mf t \oplus \mh A^1) \leftarrow \mc O (\mf t \oplus \mh A^1) \otimes_{\Ql}
\bigwedge\nolimits^1 (\mf t \oplus \mh A^1) \leftarrow \mc O (\mf t \oplus \mh A^1) \otimes_{\Ql}
\bigwedge\nolimits^2 (\mf t \oplus \mh A^1) \leftarrow \cdots
\]
It is graded so that $\mc O (\mf t \oplus \mh A^1)_d \otimes_{\Ql} \bigwedge^n (\mf t \oplus \mh A^1)$
sits in degree $d+n$. Define
\[
P^n = \mr{ind}_{\mc O (\mf t \oplus \mh A^1)}^{\mh H_{\Ql}} \big( \mc O (\mf t \oplus \mh A^1) \otimes_{\Ql}
\bigwedge\nolimits^n (\mf t \oplus \mh A^1) \big) = 
\mh H_{\Ql} \otimes_{\Ql} \bigwedge\nolimits^n (\mf t \oplus \mh A^1) .
\]
Then $P^n = \mh H_{\Ql} P^n_n$ and we have a graded projective resolution \vspace{-1mm}
\[
P^* \to \mr{ind}_{\mc O (\mf t \oplus \mh A^1)}^{\mh H_{\Ql}} (\Ql) = \mh H_{\Ql,0} .
\]
Thus $\mh H_{\Ql}$ fulfills \cite[Definition 1.1.2]{BGS} and is Koszul.\\
(b) In \cite[\S 1.2]{BGS}, the Koszul dual $E(\mh H_{\Ql})$ is defined as
$\Ext^*_{\mh H_{\Ql}} (\mh H_{\Ql,0}, \mh H_{\Ql,0})$. This is easily computed as graded vector space:
\begin{align*}
E (\mh H_{\Ql}) & = \Ext^*_{\mh H_{\Ql}} \big( \mr{ind}_{\mc O (\mf t \oplus \mh A^1)}^{\mh H_{\Ql}} \Ql,
\mh H_{\Ql,0} \big) \\
& = \Ext^*_{\mc O (\mf t \oplus \mh A^1)} \big( \Ql, \mh H_{\Ql,0} \big) \\
& = \Ext^*_{\mc O (\mf t \oplus \mh A^1)} \big( \Ql, \Ql \big) \otimes_{\Ql} \mh H_{\Ql,0} \\
& = \bigwedge\nolimits^* (\mf t \oplus \mh A^1) \otimes_{\Ql} \mh H_{\Ql,0} .
\end{align*}
Note that both $\bigwedge\nolimits^* (\mf t \oplus \mh A^1)$ and $\mh H_{\Ql,0}$ have finite dimension.
\end{proof}

The opposite algebra of $\mh H_{\Ql}$ is of the same kind, namely $\mh H_{\Ql} (G,M,q\cE^\vee)$ for 
the dual local system $q\cE^\vee$. Hence Lemma \ref{lem:3.1} also holds for $\mh H_{\Ql}^{op}$, which
means that we may use the results of \cite{BGS} with right modules instead of left modules.

Lemma \ref{lem:3.1} entails that $\mc D^b (\mh H_{\Ql} -\Mod_{\mr{fgdg}})$ admits the ``geometric
t-structure" from \cite[\S 2.13]{BGS}. Its heart is equivalent with $E(\mh H_{\Ql})-\Mod_{g}$, the
abelian category of graded right $E(\mh H_{\Ql})$-modules. Next \cite[Theorem 7.1]{Rid} shows that
\eqref{eq:3.4} sends this t-structure to the ``second t-structure" on $K^b \mr{Pure}_{G \times GL_1}
(\mf g_{N,\circ}, K_{N,\circ})$ from \cite[\S 4.2]{Rid}. In particular the heart of the second 
t-structure is equivalent with the heart of the geometric t-structure:
\begin{equation}\label{eq:3.5}
\mr{Perv}_{KD}(\mf g_{N,\circ},K_{N,\circ}) \cong E(\mh H_{\Ql})-\Mod_{g} .
\end{equation}
Let $F_\circ \in \mr{Perv}_{KD}(\mf g_{N,\circ},
K_{N,\circ})$ be the image of $E(\mh H_{\Ql})$ and let $F_s$ be the image of $F_\circ$ in
$\mc D^b_{G \times GL_1}(\mf g_{N,s}, K_{N,s})$ via \cite[Theorem 4.1]{RiRu}.
We note that by Lemma \ref{lem:3.1}.b and \eqref{eq:3.5}, 
\begin{equation}\label{eq:3.13}
\Hom_{\mc D^b_{G \times GL_1}(\mf g_{N,s},K_{N,s})} (F_s ,F_s [m]) 
\text{ is nonzero for only finitely many } m \in \Z .
\end{equation}
Choose a resolution of $E(\mh H_{\Ql})$ by free (graded right) modules of finite rank, that is 
possible by Lemma \ref{lem:3.1}.b. Via \eqref{eq:3.5}, that yields a projective resolution
\[
\cdots \to P_\circ^{-2} \to P_\circ^{-1} \to P_\circ^0 \to K_{N,\circ} 
\]
in $\mr{Perv}_{KD}(\mf g_{N,\circ},K_{N,\circ})$. Let $P_s^n$ be the image of $P_\circ^n$ in 
$\mc D^b_{G \times GL_1}(\mf g_{N,s}, K_{N,s})$. Then each $P_\circ^n$ (resp. $P_s^n$) is a 
direct sum of finitely many copies of $F_\circ$ (resp $F_s$). If $I \subset \Z$ is 
a segment such that $F_s$ lies in $\mc D^I_{G \times GL_1}(\mf g_{N,s},K_{N,s})$, then all 
$P_s^n$ belong to $\mc D^I_{G \times GL_1}(\mf g_{N,s},K_{N,s})$. This yields a chain complex
\begin{equation}\label{eq:3.6}
\cdots \to P_s^{-2} \to P_s^{-1} \to P_s^0 \to K_{N,s} , 
\end{equation}
where all objects and all morphisms come from $\mc D^I_{G \times GL_1}(\mf g_{N,s},K_{N,s})$. 
However, the entire complex is usually unbounded, because it is likely that $P_\circ^n$ and 
$P_s^n$ are nonzero for all $n \in \Z_{\leq 0}$. We define a graded algebra 
$\mc R = \bigoplus_{n \in \Z_{\geq 0}} \mc R^n$ with
\[
\mc R^n = \prod\nolimits_{k,j \in \Z_{\leq 0}} 
\Hom_{\mc D^b_{G \times GL_1}(\mf g_{N,s},K_{N,s})} (P_s^k ,P_s^j [n+k-j]) .
\]
The multiplication in $\mc R$ comes from composition in 
$\mc D^b_{G \times GL_1}(\mf g_{N,s},K_{N,s})$. For fixed $k$ and $n$, \eqref{eq:3.13} shows 
that only finitely many $j$ give a nonzero contribution to the part of $\mc R^n$
starting in $P_s^k$. This guarantees that the multiplication map $\mc R^n \times \mc R^m \to
\mc R^{n+m}$ is well-defined. 
For $M \in \mc D^b_{G \times GL_1}(\mf g_{N,s},K_{N,s})$ and $n \in \Z_{\geq 0}$ we put
\[
{\mc Hom}^n (P_s^* ,M) = \prod\nolimits_{j \in \Z_{\leq 0}} 
\Hom_{\mc D^b_{G \times GL_1}(\mf g_{N,s})} (P_s^j, M[j+n]) ,
\]
so that we obtain a functor 
\[
{\mc Hom}^* (P_s^* ,?) = \bigoplus\nolimits_{n \geq 0} {\mc Hom}^n (P_s^*,?) : 
\mc D^b_{G \times GL_1}(\mf g_{N,s},K_{N,s}) \to \mc D^b (\mc R -\Mod_{\mr{fgdg}}) .
\]
By \cite[Theorem 7.4]{Rid} and \cite[Proposition 4]{Sch}, $\mc R$ is quasi-isomorphic to its own
cohomology ring and 
\[
H^* (\mc R) \cong \End^*_{\mc D^b_{G \times GL_1}(\mf g_{N,s})} (K_{N,s}) \cong
\mh H_{\Ql} .
\]
Moreover, by \cite[Remark 7.5]{Rid} there exists a quasi-isomorphism $\mc R \to \mh H_{\Ql}$.
According to \cite[Theorem 10.12.5.1 and \S 11.1]{BeLu}, that induces an equivalence of categories
\[
\otimes^L_{\mc R} \mh H_{\Ql} :\; \mc D^b (\mc R -\Mod_{\mr{fgdg}}) \longrightarrow 
\mc D^b (\mh H_{\Ql} -\Mod_{\mr{fgdg}}) .
\]
Combining all the above, we get an additive functor
\begin{equation}\label{eq:3.7}
\otimes^L_{\mc R} \mh H_{\Ql} \circ {\mc Hom}^* (P_s^* ,?) : \; \mc D^b_{G \times GL_1}
(\mf g_{N,s},K_{N,s}) \longrightarrow \mc D^b (\mh H_{\Ql} -\Mod_{\mr{fgdg}}) .
\end{equation}
We want to show that \eqref{eq:3.7} is triangulated. An analogous statement is \cite[Lemma 7.7]{Rid}, 
which is proven in \cite[Appendix A]{Rid}. We cannot apply \cite[\S A.1]{Rid} directly to $\mh H_{\Ql}$, 
because for Hecke algebras averaging $\mc O (\mf t \oplus \mh A^1)$-module homomorphisms (between 
$\mh H$-modules) over $W_{q\cE}$ does 
not preserve the $\mc O (\mf t \oplus \mh A^1)$-linearity. Fortunately, from \cite[\S A.2]{Rid} only 
the last three lines rely on this averaging over a Weyl group. We can apply \cite[\S A.2]{Rid} with 
the group $G^\circ \times GL_1$, so that the $G^\circ \times GL_1$-equivariant cohomology of the 
variety of Borel subgroups is isomorphic to $\mc O (\mf t \oplus \mh A^1)$, like in \eqref{eq:3.9}. 
These arguments show the following: the functor \eqref{eq:3.7} sends any exact triangle in 
$\mc D^b_{G \times GL_1}(\mf g_{N,s},K_{N,s})$ to a triangle 
\begin{equation}\label{eq:3.12}
L \to M \to N \to L[1] \quad \text{in} \quad \mc D^b (\mh H_{\Ql} -\Mod_{\mr{fgdg}}),
\end{equation}
whose image in $\mc D^b (\mc O (\mf t \oplus \mh A^1) -\Mod_{\mr{fgdg}})$ is an exact triangle.

\begin{lem}\label{lem:3.9}
The triangle \eqref{eq:3.12} is already exact in $\mc D^b (\mh H_{\Ql} -\Mod_{\mr{fgdg}})$.
\end{lem}
\begin{proof}
We recall from \cite[10.12.2.9]{BeLu} that there is an equivalence of categories
\[
\mc D^b (\mh H_{\Ql} -\Mod_{\mr{fgdg}}) \cong {\mc KP} (\mh H_{\Ql}) ,
\]
where the right hand side denotes the homotopy category of $\mc K$-projective differential graded
(right) $\mh H_{\Ql}$-modules. The same holds for $\mc O (\mf t \oplus \mh A^1)$. Recall 
that the cone of a morphism $f : L \to M$ in ${\mc KP} (\mh H_{\Ql})$ is $M \oplus L[1]$ with
the differential $(d_M + f[1],-d_L)$. It comes with natural maps $\pi_1 : M \to \mr{cone}(f)$
and $\pi_2 : \mr{cone}(f) \to L[1]$. Phrased in these terms, \eqref{eq:3.12} yields a 
commutative diagram
\begin{equation}\label{eq:3.14}
\begin{array}{ccccccc}
L & \xrightarrow{f} & M & \xrightarrow{\pi_1} & \mr{cone}(f) & \xrightarrow{\pi_2} & L[1] \\
|| &  & | | & & \downarrow \phi & & || \\
L & \xrightarrow{f} & M & \xrightarrow{g_1} & N & \xrightarrow{g_2} & L[1]
\end{array},
\end{equation}
where all the modules and the horizontal morphisms belong to $\mc{KP} (\mh H_{\Ql})$ and 
$\phi$ is an isomorphism in $\mc{KP} (\mc O (\mf t \oplus \mh A^1))$. We need to prove that $\phi$ 
can be replaced by an isomorphism in $\mc{ KP} (\mh H_{\Ql})$, for then \eqref{eq:3.12} becomes 
an exact triangle.

The resolutions constructed in \cite[\S 10.12.2.4]{BeLu} show that every object of \\
$\mc D^b (\mh H_{\Ql} -\Mod_{\mr{fgdg}})$ can be represented by a free $\mh H_{\Ql}$-module, that
is, an object of $\mc{ KP} (\mh H_{\Ql})$ of the form $V \otimes \mh H_{\Ql}$ with $V$ a graded
$\Ql$-vector space (and some differential). Hence we may assume that all the objects in 
\eqref{eq:3.14} are free differential graded $\mh H_{\Ql}$-modules, say
\[
L = V_L \otimes \mh H_{\Ql},\quad M = V_M \otimes \mh H_{\Ql},\quad N = V_N \otimes \mh H_{\Ql}.
\]
With the $\mc K$-projectivity \cite[\S 10.12.2]{BeLu} we can easily compute some Hom-spaces:
\begin{align*}
& \Hom_{\mc{ KP}(\mh H_{\Ql})} (V \otimes \mh H_{\Ql}, V' \otimes \mh H_{\Ql}) \cong
\Hom_{\mc{ KP} (\Ql)} (V,V') \otimes \mh H_{\Ql} \\
& \Hom_{\mc{ KP}(\mc O (\mf t \oplus \mh A^1))} (V \otimes \mh H_{\Ql}, V' \otimes \mh H_{\Ql}) = \\
& \Hom_{\mc{ KP}(\mc O (\mf t \oplus \mh A^1))} \big( V \otimes \Ql [W_{q\cE},\natural_{q\cE}] \otimes
\mc O (\mf t \oplus \mh A^1), V' \otimes \Ql [W_{q\cE},\natural_{q\cE}] 
\otimes \mc O (\mf t \oplus \mh A^1) \big) = \\
& \Hom_{\mc{ KP} (\Ql)} (V,V') \otimes \mr{End}_{\Ql} \big( \Ql [W_{q\cE},\natural_{q\cE}] \big) \otimes
 \mc O (\mf t \oplus \mh A^1)
\end{align*}
Notice that $\mr{End}_{\Ql} \big( \Ql [W_{q\cE},\natural_{q\cE}] \big)$ is naturally 
a $W_{q\cE}$-representation, with as invariants the operators from left multiplication by elements of 
$\Ql [W_{q\cE},\natural_{q\cE}]$. Now we see that averaging over $W_{q\cE}$ provides canonical 
surjections
\begin{align*}
& \mr{Av} : \mr{End}_{\Ql} \big( \Ql [W_{q\cE},\natural_{q\cE}] \big) \to
\mr{End}_{\Ql [W_{q\cE},\natural_{q\cE}]} \big( \Ql [W_{q\cE},\natural_{q\cE}] \big) =
\Ql [W_{q\cE},\natural_{q\cE}] ,\\
& \mr{Av} : \Hom_{\mc{ KP}(\mc O (\mf t \oplus \mh A^1))} (V \otimes \mh H_{\Ql}, V' \otimes \mh H_{\Ql}) 
\to \Hom_{\mc{ KP}(\mh H_{\Ql})} (V \otimes \mh H_{\Ql}, V' \otimes \mh H_{\Ql}) 
\end{align*}
For $f' \in \Hom_{\mc{ KP}(\mc O (\mf t \oplus \mh A^1))} (V \otimes \mh H_{\Ql}, V' \otimes \mh H_{\Ql})$,
$x \in V \otimes \Ql [W_{q\cE},\natural_{q\cE}]$ and\\ 
$T \in \mc O (\mf t \oplus \mh A^1)$, this works out as
\[
\mr{Av} (f') (x T) = |W_{q\cE}|^{-1} \sum\nolimits_{w \in W_{q\cE}} f' (x N_w^{-1}) N_w T .
\]
In the same notation we consider $m = \sum_i v_i \otimes x_i T_i \in M$ and $\pi_1 (m) \in \mr{cone}(f)$.
By the commutativity of \eqref{eq:3.14}
\begin{align*}
\mr{Av}(\phi) (\pi_1 (m)) & = |W_{q\cE}|^{-1} \sum_{w \in W_{q\cE}} \sum_i \phi (\pi_1 (v_i \otimes x_i
N_w^{-1})) N_w T_i \\
& = |W_{q\cE}|^{-1} \sum_{i,w} g_1 (v_i \otimes x_i N_w^{-1}) N_w T_i \;=\;
|W_{q\cE}|^{-1} \sum_{i,w} g_1 (v_i \otimes x_i T_i ) \\
& = \sum\nolimits_i g_1 (v_i \otimes x_i T_i ) = g_1 (m) = \phi (\pi_1 (m)) .
\end{align*}
For $c = \sum_j c_j \otimes x_j T_j \in \mr{cone}(f)$ we compute
\begin{align*}
g_2 \mr{Av}(\phi) (c) & = |W_{q\cE}|^{-1} \sum_{w \in W_{q\cE}} \sum_j g_2 (\phi (c_j \otimes x_j N_w^{-1})
N_w T_j) \\
& = |W_{q\cE}|^{-1} \sum_{j,w} \pi_2 (c_j \otimes x_j N_w^{-1}) N_w T_j \;=\;
|W_{q\cE}|^{-1} \sum_{j,w} \pi_2 (c_j \otimes x_j T_j) \\
& = \sum\nolimits_j \pi_2 (c_j \otimes x_j T_j) = \pi_2 (c) = g_2 \phi (c) .
\end{align*}
The calculations show that the diagram \eqref{eq:3.14} remains commutative if we replace $\phi$ by
the $\mh H_{\Ql}$-linear map Av$(\phi)$. Finally, the proof of \cite[Proposition A.3]{Rid} 
shows that Av$(\phi)$ is an isomorphism in $\Hom_{\mc{ KP}(\mh H_{\Ql})} (\mr{cone}(f),N)$.
\end{proof}

\begin{thm}\label{thm:3.2}
Transfer the setup of Paragraph \ref{par:geomConst} to groups and varieties over an algebraically
closed field of good characteristic for $G$, and use $\Ql$ as coefficient field for all sheaves
and representations. 
\enuma{
\item The functor \eqref{eq:3.7} is an equivalence between the triangulated categories\\
$\mc D^b_{G \times GL_1}(\mf g_{N,s}, K_{N,s})$ and $\mc D^b \big( \mh H_{\Ql}(G,M,q\cE) -\Mod_{\mr{fgdg}} \big)$. 
\item There is an equivalence of triangulated categories
\[
\mc D^b_{G \times GL_1}(\mf g_{N,s}) \longrightarrow \bigoplus\nolimits_{[M,\cC_v^M,q\cE]_G}
\mc D^b \big( \mh H_{\Ql} (G,M,q\cE) -\Mod_{\mr{fgdg}} \big) .
\]
}
\end{thm}
\begin{proof}
(a) The proof of \cite[Theorem 4.3]{RiRu} explains why the arguments from \cite[\S 7]{Rid} generalize 
to our setting. These results show that \eqref{eq:3.7} commutes with the shift operator and sends 
$K_{N,s}$ to $\mh H_{\Ql}$. By Lemma \ref{lem:3.9}, the functor \eqref{eq:3.7} is triangulated.
We conclude with an application of Beilinson's lemma (in the version from \cite[Lemma 6]{Sch}).\\
(b) This follows from part (a) and Theorem \ref{thm:3.6}.
\end{proof}

Combining Theorems \ref{thm:3.11} and \ref{thm:3.2}, we have proven:

\begin{thm}\label{thm:3.3}
Assume that $G$ can be defined over a finite extension of $\Z$. 
\enuma{
\item There exists an equivalence of categories
\[
\mc D^b_{G \times GL_1}(\mf g_N, K_N) \longrightarrow 
\mc D^b \big( \mh H_{\Ql}(G,M,q\cE) -\Mod_{\mr{fgdg}} \big) ,
\]
which sends $K_N$ to $\mh H_{\Ql}(G,M,q\cE)$.

The same holds with the coefficient field $\C$ instead of $\Ql$.
\item There exists an equivalence of categories
\begin{align*}
\mc D^b_{G \times GL_1} (\mf g_N) \longrightarrow & \bigoplus\nolimits_{[M,\cC_v^M,q\cE]_G}
\mc D^b \big( \mh H (G,M,q\cE) -\Mod_{\mr{fgdg}} \big) \\
& = \; \mc D^b \big( \bigoplus\nolimits_{[M,\cC_v^M,q\cE]_G} \mh H (G,M,q\cE) -\Mod_{\mr{fgdg}} \big) .
\end{align*}
}
\end{thm}

We note that replacing $\Ql$ by the isomorphic field $\C$ is allowed because the topology of $\Ql$ 
does not play a role here. Part (a) categorifies $\mh H (G,M,q\cE)$ as differential graded algebra,
while part (b) expresses $\mc D^b_{G \times GL_1} (\mf g_N)$ as a (derived) module category.

\end{document}